\documentclass{article}
\usepackage{blindtext}
\usepackage[a4paper, total={6in, 8in}]{geometry}
\usepackage{biblatex}
\usepackage{authblk}

\usepackage[utf8]{inputenc}
\usepackage[english]{babel}
\usepackage{amsthm}
\usepackage{amssymb}
\usepackage{amsmath}
\usepackage{amscd}
\usepackage{hyperref}
\usepackage[usenames]{color}
\usepackage{csquotes}

\hypersetup{
    colorlinks=true,
    linkcolor=red,
    filecolor=magenta,      
    urlcolor=cyan,
    pdftitle={Ultralimit and Furstenberg-Poisson boundary},
    pdfpagemode=FullScreen,
    }
    
\newtheorem{theorem}{Theorem}[section]
\newtheorem{corollary}[theorem]{Corollary}
\newtheorem{lemma}[theorem]{Lemma}
\newtheorem{definition}[theorem]{Definition}
\newtheorem{example}[theorem]{Example}
\newtheorem{prop}[theorem]{Proposition}
\newtheorem{remark}[theorem]{Remark}

\newtheorem{notation}[theorem]{Notation}
\newtheorem{notationR}[theorem]{Notations and Recollections}
\newtheorem{thmx}{Theorem}
 
\newtheorem{problem}[theorem]{Problem}

\title{Entropy, Ultralimits and the Poisson boundary\footnote{The authors acknowledge the ISF support made through grant 1483/16.}}

\author[1]{Elad Sayag\footnote{elad.sayag2003@gmail.com}}
\author[2]{Yehuda Shalom \footnote{yeshalom@tauex.tau.ac.il}}
\affil[ ]{School of Mathematical Sciences, Tel-Aviv University, Israel}

\date{}

\begin{document}

\maketitle
\begin{abstract}
    
  In this paper we introduce for a group \(G\) the notion of ultralimit of measure class preserving actions of it, and show that its
  Furstenberg-Poisson boundaries can be obtained as an ultralimit of actions on itself, when equipped with appropriately chosen measures.
  We use this result in embarking on a systematic quantitative study of the basic question how close to invariant one can find measures on a \(G\)-space, particularly for the action of the group on itself.
  As applications we show that on amenable groups there are always \enquote{almost invariant measures} with respect to the information theoretic Kullback-Leibler divergence (and more generally, any \(f\)-divergence), making use of the existence of measures with trivial boundary. More interestingly, for a free group \(F\) and a symmetric measure \(\lambda\) supported on its generators, one can compute explicitly the infimum over all measures \(\eta\) on \(F\) of the Furstenberg entropy \(h_{\lambda}(F,\eta)\). Somewhat surprisingly, while in the case of the uniform measure on the generators the value is the same as the Furstenberg entropy of the 
Furstenberg-Poisson boundary of the same measure \(\lambda\), in general it is the Furstenberg entropy of the 
Furstenberg-Poisson boundary of a measure on \(F\) \textbf{different} from \(\lambda \).



\end{abstract}

\section{Introduction}
This paper has two main goals. 
The first is to introduce the problem of minimizing the Furstenberg-entropy for group actions. We relate this problem to amenability and boundary actions.\\
The second goal is to introduce the methods of ultralimits into ergodic theory of group actions by providing a construction of the Furstenberg-Poisson boundary via ultralimits. This construction sheds new light on known properties of the Furstenberg-Poisson boundary (see in particular the ending paragraph of the introduction for a very recent use of it in the non-commutative setting).\\
We use this construction and basic properties of ultralimits to establish the connection between the problem of minimizing entropy to both amenability and boundary actions.

\subsection{Minimizing entropy problem}
Many natural actions of discrete groups \(G\curvearrowright X\), both in the topological and in the ergodic theoretic setting, admit no invariant probability measure. Basic examples to keep in mind are the action of an infinite discrete group on itself, the action of the free group \(F_d\) on the space \(\partial F_d\) of infinite reduced words\footnote{More generally, the action of a hyperbolic group on its Gromov boundary.} and the action of \(SL_{n}(\mathbb{Z})\) on \(\mathbb{P}_{\mathbb{R}}^{n-1}\) or on \(SL_n(\mathbb{R})/P(\mathbb{R})\) for any parabolic subgroup\footnote{More generally the action of any discrete Zariski-dense subgroup of a real semi-simple Lie group with no compact factors on a flag space. For the non-existence of invariant measure see  \cite{Zimmer1984ErgodicTA} or \cite{furstenberg1976note}.}.
One can ask the following informal basic question:
\[
\text{{\it How close to invariant can measures on \(X\) get?}}
\]
The notion of Furstenberg entropy, which we now recall, enables one to formalize this question. \\
A standard method of measuring distance between two measures is the Kullback–Leibler (KL) divergence (see chapter 2 in the book \cite{cover1999elements}), which is
an extremely important and well known information theoretic notion. Recall that given two probability measures \(\nu,m\) on \(X\), their KL-divergence is defined by:
\[D_{KL}(m||\nu):=\intop_{X} -\ln(\frac{d\nu}{dm})dm\]
The KL-divergence is always non-negative, and it is zero exactly when \(\nu=m\). The KL-divergence is much more sensitive than the norm distance \(||\nu-m||\) and is (defined to be) infinite if \(\nu,m\) are not in the same measure class, however it can be infinite even if they are.
With the help of this notion one can quantify the lack of invariance of a given measure on \(X\) under an action of a group \(G\). More specifically, fix a measure \(\lambda\) on \(G\) (whose 
support generates \(G\)), and consider the associated Furstenberg entropy for \textbf{any} (quasi-invariant) probability measure \(\nu\) on  \(X\):
\[h_{\lambda}(X,\nu)=h_{\lambda}(\nu):=\sum_{g\in G} \lambda(g)D_{KL}(g\nu||\nu)=\sum_{g\in G}\lambda(g)\intop_{X} -\ln(\frac{dg^{-1}\nu}{d\nu})\: d\nu\]
We emphasize that we consider here the Furstenberg entropy of \textbf{all} probability measures \(\nu\) and not necessarily 
\(\lambda\)-stationary measures (i.e., \(\lambda\ast\nu=\nu\)), as was previously done in the literature.
The Furstenberg entropy vanishes exactly for invariant measures and serves as a natural quantification for non-invariance. One can also define a similar notion for the whole
family of the well known information theoretic \enquote{distance functions} called \(f\)-divergence, which our results cover as well. \\
Our search for the most invariant measure(s) on \(X\) can now be formalized by defining:
\[I_{\lambda}(X):=\inf_{\nu\in M(X)} h_{\lambda}(X,\nu)\]
Of course, one needs to be more precise here about the class \(M(X)\) of measures \(\nu\) we take the infimum over. Typically it shall be a fixed measure class (consisting of all fully supported measures when \(X\) is discrete). In the case of a topological action, we consider the infimum to be over all the Borel measures on \(X\) and denote the result by \(I_{\lambda}^{top}(X)\). We note that in this case, if \(X\) is compact then the infimum is achieved (see the discussion in \ref{subsection Entropy on other actions}).\\
A basic fact (see Corollary \ref{Cor Group-Space-Entropy}) is that \(I_{\lambda}(X)\leq I_{\lambda}(G)\) for any \(G\)-space \(X\), which motivates our initial focus on the value of the latter. \\
These entropy-theoretic quantities give rise to many natural questions, let us state a few of them: 
\newpage
\begin{problem}
\(\) 
\begin{itemize}
    \item
    Find a method for calculating \(I_{\lambda}(X)\) or \(I_{\lambda}^{top}(X)\) in various cases (such as the examples at the beginning of the section).
    \item In the case of \(I_{\lambda}^{top}(X)\) what can be said about a minimizing measure? is it unique? (Somewhat surprisingly, one of our findings is that a minimizing measure need not be \(\lambda\)-stationary.)
    \item What can be said about actions for which which \(I_{\lambda}(X) = I_{\lambda}(G)\) or \(I_{\lambda}(X) = 0\) ?
    \item In the cases of action of lattices on flag spaces calculate \(I_{\lambda}(X)\) for the canonical measure class (e.g. the Haar measure for the maximal compact subgroup) and appropriate \(\lambda\)'s. Is the infimum attained? Is it the same as \(I_{\lambda}^{top}(X)\)? 
\end{itemize}
\end{problem}
This problem (especially the second one) arose naturally from an attempt to tackle the well known problem of dependence of the Liouville property on the finite symmetric generating measure, see the discussion in subsection \ref{subsection Possible relation with the Liouville problem}.\\
In this paper, we are able to calculate \(I_{\lambda}(G)\) for amenable groups and free groups.\\
The main tool is the new construction of the Furstenberg-Poisson boundary, Theorem \ref{Theorem E}.\\
For amenable groups we use the existence of a Liouville measure on amenable groups (see \cite{BoundaryEntropy} or \cite{Rosenblatt1981ErgodicAM}) to conclude
\begin{thmx}[See Theorem \ref{Theorem application for amenability}]\label{Theorem A}
If \(G\) is amenable and \(\lambda\) has finite  support then \(I_{\lambda}(G)=0\).\\
Namely, on any amenable group there are KL-almost invariant measures.
\end{thmx}
It should be noted that the converse to this result is easy (and follows immediately from Pinsker's inequality-- Theorem 4.19 in \cite{Boucheron2013ConcentrationI}). 
We will deduce from Theorem \ref{Theorem A} an interesting fixed point property for amenable groups acting on a certain distinguished topological model of their Furstenberg-Poisson boundary (see Theorem \ref{Theorem fixed point for stationary actions}) -- as it turns out (see Corollary \ref{corollary RN model for Furstenberg-Poisson boundary}) every countable group admits a canonical one. This theorem was proven in \cite{FixedPointAmenable} by entirely different method. Further study of the action on this canonical model seems interesting, in particular seeking a converse to our result.\\
We now formulate our results regarding \(F_d\) -  the free group on \(d\) generators.
\begin{thmx}[See Corollary \ref{Corollary case of uniform measure}]\label{Theorem B}
Let \(\lambda\) be the symmetric uniform measure on the free generators and their inverses. Then
\[I_{\lambda}(F_d)=\frac{d-1}{d}\ln(2d-1)\]
\end{thmx}
In fact, we can compute \(I_{\lambda}(F_d)\) for any measure \(\lambda\) which is symmetric and supported on the free generators and their inverses:
\begin{thmx}[See Theorem \ref{Theorem entropy for F action on F}]\label{Theorem C}
Let \(\lambda\) be a symmetric measure on \(F_d\) supported on the free generators and their inverses. Then there is a symmetric measure \(\mu\), generally different from \(\lambda\), supported on the generators and their inverses, with:
\[I_{\lambda}(F_d)=h_{\lambda}(\partial F_d, \nu_{\mu})\]
where \(\nu_{\mu}\) is the the harmonic (hitting, stationary) measure for the \(\mu\)-random walk on the standard topological boundary of \(\partial F_d\).
\end{thmx}
An explicit formula for the measure \(\nu_{\mu}\) and a complete description of the correspondence \(\lambda\mapsto\mu\) are given in subsections \ref{subsection The Boundary of the Free group and harmonic measures} and \ref{subsection Measures on the boundary minimizing in their measure class} respectively. \(\)\\
More generally, one can use the well known information theoretic generalization \(f\)-divergence in place of the \(KL\)-divergence. Namely, we introduce
\[h_{\lambda,f}(\nu)=\sum_{g\in G} \lambda(g)D_{f}(g\nu||\nu)\]
and a corresponding \(I_{\lambda,f}\). Our results are formulated in this setting. For example, our generalization of Theorem \ref{Theorem B} states:
\begin{thmx}[See Corollary \ref{Corollary case of uniform measure}]\label{Theorem D}
Let \(\lambda\) be the symmetric uniform measure on the free generators of \(F_d\) and their inverses. Then:
\[I_{\lambda,f}(F_d)=\frac{2d-1}{2d}f(\frac{1}{2d-1})+\frac{1}{2d}f(2d-1)\]
\end{thmx}
Another (simple, but important) ingredient in the proof of Theorems \ref{Theorem B},\ref{Theorem C},\ref{Theorem D} will be to give a general criterion\footnote{It is important to note that this criterion is only about entropy in a fixed measure class. Hence, it could not be used for calculating \(I_{\lambda,f}^{top}(X)\) (see the discussion in subsection \ref{subsection Entropy on other actions}).} (see Proposition \ref{Proposition f-entropy minimize}) for a measure \(\nu\) to satisfy \(h_{\lambda,f}(X,\nu)=I_{\lambda,f}(X)\).

\subsection{Furstenberg-Poisson boundary as an Ultralimit}
The Furstenberg-Poisson boundary \(\mathcal{B}(G,\mu)\) of a group \(G\) with a probability measure $\mu$ was first introduced by Furstenberg in the papers \cite{PoissonFormulaFurstenberg},\cite{Furstenberg1973BoundaryTA} and since then it was studied extensively (cf. \cite{BoundaryEntropy},\cite{FurstenbergGlasner},\cite{BaderShalom}, \cite{kaimanovichhyperbolicproperties}).\\ 
We remind here that the Furstenberg-Poisson boundary is a probability space \((B,\nu)\) with a measurable group action \(G\curvearrowright B\) such that the measure is stationary, e.g. \(\mu\ast\nu = \nu\). 
\(\)\\
In this paper, we give a new construction of the Furstenberg-Poisson boundary via ultralimits.
The objects of interest for us are probability spaces \((X,\nu)\) together with a measurable \(G\)-action preserving the measure class. 
\(\)\\
Our approach towards defining ultralimits of such actions will be via ultralimits of \(C^{*}\)-algebras as in  \cite{Heinrich1980UltraproductsIB} (see Remark \ref{remark ultralimits are via algebras}). More precisely, we define the ultralimit \(\mathcal{U}\lim_{\text{big}} (X_i, \nu_i)\) of \enquote{uniformly bounded-quasi-invariant} \(G\)-spaces \((X_i,\nu_i)\) (see Definition \ref{Definition ultralimit of G measure spaces}).\\
Recall that given a \(G\)-space \((X,\nu)\) it admits a minimal measure preserving factor, called
the measurable Radon-Nikodym factor (RN for short). As we observe (see Proposition \ref{Proposition RN model for RN factor}) the measurable RN-factor always admits a canonical topological model (that is, a realization on a compact metrizable space on which the group acts continuously). We denote this space by \((X,\nu)_{RN}\) -- it is the smallest model for which the Radon-Nikodym derivatives \(\frac{dg\nu}{d\nu}\) are continuous.
\(\)\\
The RN factor of our large ultralimit space is denoted by \(\mathcal{U}{\lim}_{RN} (X_i,\nu_i)\).
\(\)\\
With this notation, we can state our new construction of the Furstenberg-Poisson boundary as an ultralimit of the action of the group on itself (taken with a sequence of \enquote{Abel} measures):
\begin{thmx}[See Theorem \ref{Theorem Abel space is Furstenberg-Poisson boundary}]\label{Theorem E}
Let \(G\) be a discrete countable group, and \(\mu\) a generating measure on \(G\). Denote by \(\mu_a=(1-a)\sum_{n=0}^{\infty} a^{n}\mu^{*n}\). Fix \(\mathcal{U}\) an ultrafilter  on \((\frac{1}{2},1)\) converging to 1.\\
Then \(\mathcal{U}{\lim}_{RN} (G,\mu_a)\) is a topological model for the Furstenberg-Poisson boundary \(\mathcal{B}(G,\mu)\).
\end{thmx}
\(\)\\
From this perspective, many classical claims on the Furstenberg-Poisson boundary become clear:\\
1. The maximality of the Furstenberg-Poisson boundary among stationary actions, proved in \cite{FurstenbergGlasner}. Actually, our proof of Theorem \ref{Theorem E} is by showing maximality for our construction (see Proposition \ref{Proposition universal property of abel}). \\
2. Kaimaniovich-Vershik basic entropy formula (see \cite{BoundaryEntropy}, Corollary \ref{corollary entropy KV using ultralimits}): \(h_{\mu}(\mathcal{B}(G,\mu)) = \lim_{n\to\infty} \frac{H(\mu^{*n})}{n}\). \\
3. Weak containment of the quasi-regular-\(L^{2}\) representation in the regular representation of the group (see \cite{weakcont}).\\
\(\)\\
A novel aspect of of Theorem \ref{Theorem E} is that one can get quantitative results on the Furstenberg-Poisson boundary relating it to the action of the group on itself (see Corollary \ref{Corollary RN parameters on group and boundary}) without any assumption on the measure \(\mu\). This is the main tool in the proof of Theorems \ref{Theorem A}-\ref{Theorem D}. We remark (see Remark \ref{remark martin boundary}) that while it may be possible to prove Corollary \ref{Corollary RN parameters on group and boundary} via other more ad-hoc methods, we believe that the tools of ultralimits could find use in various applications. The use of ultralimits in the sequel paper \cite{ultrAmenable} allow us to handle the topological case, in particular determining \(I_{\lambda,f}^{top}(\partial F_d)\).\\
\(\)\\
Finally, we note that recently the ideas around Theorem \ref{Theorem E} found another interesting application. In the paper \cite{das2022poisson} by Das and Peterson, they introduce a Poisson boundary in a non-commutative setting. They initiate entropy theory of non-commutative boundaries, and proved analogies of part of Kaimanovich-Vershik’s fundamental theorems regarding entropy \cite{BoundaryEntropy} while the rest are left open. Generalizing our approach to the non-commutative setting, S. Zhou \cite{zhou2023noncommutative} recently completed the proof of those fundamental theorems.




\section{Preliminaries}\label{section Preliminaries}

\subsection{G-spaces and Radon-Nikodym factor}\label{subsection G-spaces}

Throughout this paper, \(G\) is a discrete countable group.
\begin{definition}
\begin{itemize}
    \item
    A \emph{\textbf{QI} (quasi-invariant) \textbf{\(G\)-space}} is a probability space \(\mathcal{X}=(X,\nu)\) together with a measurable action of \(G\) on \(X\), so that \(\nu\) is quasi-invariant (i.e., \(g\nu,\nu\) are in the same measure class).
    \item
    A \textbf{\emph{factor map}} between  QI \(G\)-spaces \(\mathcal{X},\mathcal{Y}\) is a measurable mapping \(T: X\to Y\) which is \(G\)-equivariant (e.g. \(T(gx)=gT(x)\) a.e.) and with \(T_{*}\nu_X=\nu_Y\). One says \(\mathcal{Y}\) is a factor of \(\mathcal{X}\), and that \(\mathcal{X}\) is an extension of \(\mathcal{Y}\).
    \item
    We say the \(T\) is a \textbf{\emph{measure preserving extension}} if \(\frac{dg\nu_{Y}}{d\nu_{Y}}\circ T=\frac{dg\nu_{X}}{d\nu_{X}}\).
    \item
    We say \(\mathcal{X}\) is \emph{\textbf{BQI} (bounded quasi invariant)} if \(R_\nu(g)=\frac{dg\nu}{d\nu}\in L^{\infty}(\mathcal{X})\) for any \(g\in G\).\\
    In this case, a \textbf{\emph{majorant}} for \(\mathcal{X}\) is a function \(M:G\to\mathbb{R}\) such that  \(||\ln R_\nu(g)||_{\infty}\leq M(g)\).
\end{itemize}
\end{definition}
\begin{remark}
Sometimes, when we write QI-\(G\)-space, we wouldn't like to fix the probability measure on \(X\), only the measure class. This should not 
give rise to any confusion.
\end{remark}
In this paper we shall inherently need to work also with non-standard measure spaces, i.e., ones that are not countably separated. However, we will usually \enquote{reduce} the \(\sigma\)-algebra to a countably generated one. We then would like to replace the measure space with the new \(\sigma\)-algebra by a standard measure space. In general, this replacement will not be by an isomorphism in the sense that there are inverse factors, but it will be in the sense that there is an isomorphism of the \(L^{\infty}\) spaces.\\
In order to show that such isomorphism always comes from a factor map, we use the following well known (see Theorem 2.1 in \cite{RAMSAY1971253})
\begin{lemma}\label{lemma mapping in L infinity induces factor}
Suppose \(\Omega\) is a measure space, and \(A\subset L^{\infty}(\Omega)\) is a separable sub \(C^{*}\)-algebra. Consider the Gelfand dual \(X\) of \(A\). Then there is a measurable mapping \(\pi:\Omega\to X\) that induces the embedding \(A\to L^{\infty}(\Omega)\). Such a mapping \(\pi\) is unique up to equality a.e.\\
If \(\Omega\) is a QI \(G\)-space and \(A\subset L^{\infty}(\Omega)\) is \(G\)-invariant, then \(\pi\) is a \(G\)-factor.
\end{lemma}

\begin{definition}
A topological model for a \(QI\) \(G\)-space \(\mathcal{X}\) is a \(QI\) \(G\)-space \(\mathcal{Y}=(Y,\nu)\) together with a factor \(\pi: \mathcal{X}\to\mathcal{Y}\) such that:
\begin{enumerate}
    \item  \(Y\) is a compact metrizable (equivalently compact Hausdorff and second countable) topological space with a continuous \(G\)-action, the \(\sigma\)-algebra on \(\mathcal{Y}\) is the Borel \(\sigma\)-algebra.
    \item \(\pi\) induces an isomorphism \(L^{\infty}(\mathcal{Y})\cong L^{\infty}(\mathcal{X})\)
\end{enumerate}
\end{definition}
From Lemma \ref{lemma mapping in L infinity induces factor} we see that if \(L^{\infty}(\mathcal{X})\) is separable with respect to the weak topology (that is, the \(\sigma\)-algebra is generated by countably many sets) then \(\mathcal{X}\) posses a topological model. Moreover, all of them are isomorphic when considered as \(QI\)-\(G\)-spaces factors of \(\mathcal{X}\).\\
\(\)\\
Let us recall the definition of Radon-Nikodym factor, which using Lemma \ref{lemma mapping in L infinity induces factor} is easily seen to agree with the definition given in the introduction
\begin{definition}
Let \(\mathcal{X}=(X,\Sigma,\nu)\) be a QI \(G\)-space. Its Radon-Nikodym factor is defined to be the QI \(G\)-space which is a topological model of \((X,\Sigma_{RN},\nu)\) where \(\Sigma_{RN}\) is the \(\sigma\)-algebra generated by \(\{\frac{dg\nu}{d\nu}\big| g\in G\}\).
\end{definition}
The next definition and proposition formalizes the observation that the classical Radon-Nikodym factor has a canonical topological model. This is easily seen from the perspective of \(C^*\)-algebras.

\begin{definition}
A \textbf{\emph{RN model}} of a BQI \(G\)-space \(\mathcal{X}\) is a compact metrizable space \(X\) with a Borel probability measure \(\nu\) and a continuous \(G\)-action on \(X\) such that:
\begin{itemize}
    \item \((X,\nu)\) is isomorphic to \(\mathcal{X}\) as BQI \(G\)-spaces.
    \item For all \(g\in G\) one has \(\frac{dg\nu}{d\nu}\in C(X)\).
\end{itemize}
\end{definition}

\begin{prop}\label{Proposition RN model for RN factor}
Let \(\mathcal{X}\) be a BQI \(G\)-\(C^*\)-space and consider its (measurable) RN factor \(\mathcal{X}_{RN}\). Then it has a topological model which is a RN model. In fact, it has a canonical such model \(X_{RN}\) which is a continuous factor of all other RN models.
\end{prop}
\begin{proof}
Consider the closure of the algebra generated by all Radon-Nikodym derivatives \(A_{RN}:=C^{*}\big(R_{\nu}(g)=\frac{dg\nu}{d\nu}\big|\:g\in G\big)\subset L^{\infty}(\mathcal{X})\). Note that \(A_{RN}\) is separable and it is \(G\)-invariant since \(g\cdot R_\nu(h)=\frac{R_\nu(gh)}{R_\nu(g)}\). Let \(X_{RN}\) be its Gelfand dual. The space \(X_{RN}\) is a RN-model. Namely, it is a compact metrizable space with a continuous action of \(G\) and a BQI Borel measure \(\nu\) such that \(\frac{dg\nu}{d\nu}\in C(X_{RN})\cong A_{RN}\) for all \(g\in G\).\\
By Lemma \ref{lemma mapping in L infinity induces factor} applied to \(A_{RN}\subset L^{\infty}(\mathcal{X}_{RN})\) we obtain a factor of \(G\)-spaces \(\pi: \mathcal{X}\to X_{RN}\). Via the map \(\pi\), the space \(X_{RN}\) is a topological model of \(\mathcal{X}_{RN}\).\\
To show the universality of \(X_{RN}\), let \(Y\) be a topological model of \(\mathcal{X}_{RN}\) which is a RN-model. Consider the inclusions \(C(Y)\subset L^{\infty}(Y)\cong L^{\infty}(\mathcal{X}_{RN})\subset L^{\infty}(X)\). As for every \(g\in G\) we have \(\frac{dg\nu}{d\nu}\in C(Y)\), we conclude \(A_{RN}\subset C(Y)\). This embedding induces a continuous factor \(Y\to X_{RN}\).
\end{proof}
\begin{remark}
Note that \(X_{RN}\) has the property that the RN-cocyle separates the points of \(X\). This characterizes \(X_{RN}\) among the RN models.
\end{remark}

We finish this section with the following observation regarding the Furstenberg-Poisson boundary.
\begin{corollary}\label{corollary RN model for Furstenberg-Poisson boundary}
Suppose \(G\) be a discrete countable group. Then any BQI \(G\)-space which is equal to its (measurable) RN-factor has a topological model which is RN and minimal among such topological models.\\
In particular, this applies to the Furstenberg-Poisson boundary, hence in particular the latter admits a \enquote{canonical} topological model.   \(\mathcal{B}(G,\mu)\) for a generating measure \(\mu\).
\end{corollary}
\begin{proof}
The first part is clear by Proposition \ref{Proposition RN model for RN factor}. The second part is by noting that \(\mathcal{B}(G,\mu)\) is its own measurable Radon-Nikodym factor. Indeed, the linear span of the RN cocyle is dense in \(L^{1}\), by lemma 2.2 in \cite{5remarksPoisson} (This lemma is a trivial application of Hahn-Banach and the fact that Poisson integral is an isomorphism on the Furstenberg-Poisson boundary).
\end{proof}
The interest in this model of the Furstenberg-Poisson boundary is illustrated in Theorem \ref{Theorem fixed point for stationary actions}, as it is shown that for amenable groups, and symmetric \(\mu\), there is always a unique fixed point there.

\subsection{G-spaces in the language of 
C*-algebras}
\subsubsection{C*-algebras and probability spaces}\label{subsubsection C*-algebras and probability spaces}
A basic tool in our work will be commutative \(C^*\) algebras. We review some basic notions from this theory that arise from translating the relevant concepts of measure theory to algebra.
\begin{notationR}
Let \(A\) be a commutative \(C^*\) algebra with unit 1. 
\begin{itemize}
    \item
    \(A_+\) denotes the \emph{positive elements} of \(A\).
    \item
    A \emph{positive functional/measure} \(\nu\) is an element \(\nu\in A^*\) with \(\nu(A_+)\subseteq \mathbb{R}_{\geq 0}\).
    \item
    If a measure \(\nu\) satisfies in addition that \(\nu(1)=1\) we say \(\nu\) is a \emph{state/probability measure}.
    \item
    We say that a measure \(\nu\) is \emph{faithful} if: \(\forall a\in A_+:\:\: \nu(a)=0\implies a=0\).
    \item
    Consider the action of \(A\) on \(A^*\) by \((a\cdot \nu)(b)=\nu(ba)\).\\
    If \(a\in A_+\) and \(\nu\) is a measure then so is \(a\cdot \nu\). \\
    If \(\nu\) is a faithful measure then we have \(a\in A,\: a\cdot \nu=0\implies a=0\).
    \item
    For a general measure \(\nu\) we have \(\mathcal{N}_\nu:=\big\{a\in A\big| \nu(a^*a)=0\big\}=\big\{a\in A \big|a\cdot \nu=0 \big\}\). Indeed this follows from Cauchy-Schwartz inequality for \(\langle a,b \rangle_\nu=\nu(b^*a)\).
    \item
    \(\mathcal{N}_{\nu}\) is an ideal of \(A\) and \(\nu\) induces a measure \(\overline{\nu}\) on \(\frac{A}{\mathcal{N}_\nu}\) which is faithful.\\
    We call \((\frac{A}{\mathcal{N}_\nu},\overline{\nu})\) the \textbf{\emph{reduced version}} of \((A,\nu)\) and denoted it by \((A,\nu)_{red}\). 
    \item
    We say that a measure \(m\) is \textbf{\emph{bounded-absolutely-continuous}} with respect to a measure \(\nu\) and denote it by \(m\ll^b\nu\) iff there is \(c\in A_+\) with \(m=c\cdot \nu\). \\
    In this case we have that \(\mathcal{N}_\nu\subset \mathcal{N}_m\) and hence \(m\) defines a measure on \(\frac{A}{\mathcal{N}_\nu}\).
    \item
    If \(\nu\) is faithful and \(m\ll^b\nu\) then there exists a unique \(c\in A_{+}\) with \(m=c\cdot\nu\). This element \(c\) is called the \textbf{\emph{Radon-Nikodym derivative}}. It is denoted by \(c=\frac{dm}{d\nu}\).\\
    If in addition we have \(\nu\ll^b m\) then \(m\) is faithful and \(\frac{d\nu}{dm}=(\frac{dm}{d\nu})^{-1}\).
\end{itemize}
\end{notationR}
Now we define the notion of probability space in the realm of \(C^*\)-algebras
\begin{definition}\label{definition C probability space}
\(\)
\begin{itemize}
    \item A \textbf{\emph{\(C^*\)-probability space}} is a pair \((A,\nu)\) of a unital commutative \(C^*\) algebra together with a faithful state \(\nu\).
    \item
    A \textbf{\emph{factor map}} \(T:\mathcal{B}\to\mathcal{A}\) between \(C^*\)-probability spaces \(\mathcal{B}=(B,m)\) and \(\mathcal{A}=(A,\nu)\) is a \(*\)-homomorphism \(T:B\to A\) such that \(T^*\nu=m\). We say that \(\mathcal{B}\) is a factor of \(\mathcal{A}\) and \(\mathcal{A}\) is an extension of \(\mathcal{B}\). 
\end{itemize}
\end{definition}
Note that in the definition of a factor, \(T\) is automatically an injection and thus an isometry.

\subsubsection{G-C*-spaces}\label{subsubsection G-C*-spaces}
Let \(A\) be a \(C^*\) algebra. By an \textbf{action} of \(G\) on \(A\) we will mean a left action such that for any \(g\in G\), the mapping \(a\mapsto g\cdot a={}^ga\) is a \(C^*\)-algebra isomorphism.
Such action induces an action \(G\curvearrowright A^{*}\) by \({}^g\nu=g\cdot\nu=(g^{-1})^*\nu\). The set of states is \(G\)-invariant.\\
In addition, any  finite measure \(\mu\) on \(G\) acts on \(A\) by: \(\mu\ast a=\sum_{g} \mu(g)\cdot {}^ga\). Similarly \(\mu\) acts on \(A^{*}\).\\
We now define the notion of \(G\)-space in the context of \(C^*\)-algebras
\begin{definition}\label{definition BQI G C space}
A \textbf{\emph{BQI \(G\)-\(C^*\)-space}} is a \(C^*\)-probability space \(\mathcal{A}=(A,\nu)\) together with an action of \(G\) on \(A\) such that for any \(g\in G\) we have \({}^g\nu\ll^{b} \nu\).\\
We will denote \(R_{\mathcal{A}}(g)=R_\nu(g):=\frac{dg\nu}{d\nu}\), the \textbf{\emph{Radon-Nikodym cocyle}}.\\
A \textbf{\emph{majorant}} for \(\mathcal{A}\) is a function \(M:G\to\mathbb{R}\) such that  \(||\ln(R_\nu(g))||\leq M(g)\).
\end{definition}    
\begin{remark}
It is easy to see that \(R_{\nu}(gh)={}^gR_\nu(h)\cdot R_\nu(g)\), so \(R_\nu:G\to A^{\times}\) is indeed a cocyle.
\end{remark}

\begin{remark}
Note that it is not always true that for any probability measure \(\mu\) on \(G\) we have \(\mu\ast\nu\ll^{b} \nu\). However, the following is true (and trivial):\\
Suppose \(G\) is a group that acts on a \(C^*\)-algebra \(A\) and \(\nu\) is a state on \(A\) such that \(\mathcal{N}_\nu\) is \(G\)-invariant. Then for any  complex finite measure \(\mu\) on \(G\) we have \((\mu\ast\nu)(\mathcal{N}_\nu)=0\) and thus \(\mu\ast \nu\) is well defined state on \((A,\nu)_{red}\), that agrees with \(\mu\ast \overline{\nu}\).\\
This observation will be used without mentioning in subsection \ref{subsection Ultralimit of C*-probability spaces and the G-equivariant case}.
\end{remark}

\begin{definition}
A \textbf{\emph{factor}} \(p:\mathcal{B}\to \mathcal{A}\) between two BQI \(G\)-\(C^*\)-spaces \(\mathcal{A},\mathcal{B}\) is a factor of \(C^*\)-probability spaces that is \(G\)-equivariant. We say that \(\mathcal{B}\) is a factor of \(\mathcal{A}\) and \(\mathcal{A}\) is an extension of \(\mathcal{B}\).\\
The factor \(p\) is called \textbf{\emph{measure preserving extension/factor}} if \(p\big(R_{\mathcal{B}}(g)\big)=R_{\mathcal{A}}(g)\).
\end{definition}
One can define RN-factor in this setting

\begin{definition}
The \textbf{\emph{Radon Nikodym factor}} of a BQI \(G\)-\(C^*\)-space \(\mathcal{A}\) is the BQI-\(G\)-\(C^*\)-space given by \(\mathcal{A}_{RN}=(A_{RN},\nu|_{A_{RN}})\) where \(A_{RN}=C^{*}(R_{\nu}(g)\big|g\in G)\subset A\) is the sub-\(C^*\)-algebra generated by the Radon-Nikodym cocyle.
\end{definition}
Note that \(A_{RN}\) is \(G\)-invariant as \(^{g}R_\nu(h)=\frac{R_\nu(gh)}{R_\nu(g)}\), and also \({}^g\nu|_{A_{RN}}\ll^{b}\nu|_{A_{RN}} \), so indeed \(\mathcal{A}_{RN}\) is a BQI \(G\)-\(C^*\)-space.
The factor map \(p:\mathcal{A}_{RN}\to \mathcal{A}\) is a measure preserving extension, and if \(\pi:\mathcal{B}\to\mathcal{A}\) is a measure-preserving extension, then \(p\) factors through \(\pi\).\\
\(\)\\
We end this section by noting the relation between our definitions for \(G\)-spaces in the algebraic and classical setting.\\
In one direction, let \((X,\nu)\) be a BQI \(G\)-space, then \((L^{\infty}(X),\nu)\) is a BQI \(G\)-\(C^*\)-probability space. \\ 
In the opposite direction, for a BQI \(G\)-\(C^*\)-probability space \((A,\nu)\), Gelfand's representation theorem provides a compact Hausdorff space \(X\) with \(C(X)\cong A\). The state \(\nu\) defines a Borel measure on \(X\). Moreover, \(X\) admits a natural continuous \(G\)-action and with those, \((X,\nu)\) is a BQI \(G\)-space. Also, the Radon Nikodym cocyle consists of continuous functions. Note that when \(A\) is a separable \(C^*\)-algebra, \(X\) is a RN-model.

\section{f-divergence and Entropy}\label{section f-divergence and Entropy}
Throughout this section, \(f\) denotes a convex function in \((0,\infty)\) with \(f(1)=0\).\\
We denote by \(F(x,y):=yf(\frac{x}{y})\) the corresponding homogeneous function in two variables.\\
It is well known that \(F:(0,\infty)^2\to\mathbb{R}\) is convex (e.g. see Lemma \ref{Lemma F is convex}).\\
In this section we define the notion of entropy based on the notion of \(f\)-divergence, generalizing the Furstenberg entropy in the case
of \(f(t)=t\ln(t)\), and establish a criterion for a measure to minimize it.

\subsection{Definition of f-entropy}\label{subsection Definition of f-entropy}
Recall the definition of \(f\)-divergence:
\begin{definition}[$f$-divergence]
Let \(\eta,m\) be two measures on a measurable space \(X\) in the same measure class. We say that \(\eta\) has finite \(f\)-divergence with respect to \(m\) if \(f(\frac{d\eta}{dm})\in L^{1}(m)\) and define:
\[D_f(\eta||m):=\intop_{X} f(\frac{d\eta}{dm})\:dm\]
\end{definition}
For a more symmetric formulation, note that we can take any measure \(\omega\) in the same measure class then the \(f\)-divergence is finite iff \(F(\frac{d\eta}{d\omega},\frac{dm}{d\omega})\in L^{1}(\omega)\) and we have:
\[D_f(\eta||m)=\intop_{X} F(\frac{d\eta}{d\omega},\frac{dm}{d\omega})\:d\omega\]
The \(f\)-divergence does not change when we change \(f\) by a linear function vanishing at \(1\). Thus, by the convexity one can always assume \(f\geq0\). Hence, we define \(D_f\) to be \(+\infty\) when \(f(\frac{d\eta}{dm})\notin L^1(m)\).\\
First we have:
\begin{lemma}\label{Lemma F is convex}
Let \(f:(0,\infty)\to\mathbb{R}\) be a convex function, then \(F(x,y)=yf(\frac{x}{y})\) is convex.\\
Moreover, if \(f\) is strictly convex and \(x_0,x_1,y_0,y_1\in(0,\infty)\) and \(\frac{x_0}{y_0}\neq \frac{x_1}{y_1}\) then \(F\) restricted to the interval between \((x_0,y_0),(x_1,y_1)\) is strictly convex.  
\end{lemma}
\begin{proof}
Note that the convexity condition 
\[F((1-t)x_0+tx_1,(1-t)y_0+ty_1)\leq (1-t)F(x_0,y_0)+t F(x_1,y_1)\]
is the same as \(f\big(\frac{(1-t)x_0+tx_1}{(1-t)y_0+ty_1}\big)\leq f(\frac{x_0}{y_0})\cdot\frac{(1-t)y_0}{((1-t)y_0+ty_1)}+f(\frac{x_1}{y_1})\cdot\frac{ty_1}{((1-t)y_0+ty_1)}\) which follows from the convexity of \(f\). The moreover part follows too.
\end{proof}
The previous lemma yields the following using (conditional) Jensen inequality:
\begin{lemma}\label{Lemma f-divergence}
\(f\)-divergence has the following properties:
\begin{enumerate}
    \item 
    \(D_f\) is a convex function in both variables and non-negative.
    \item
    For a measurable mapping \(\pi: X\to Y\) we have: 
    \[D_f(\pi_*\eta||\pi_*m)\leq D_f(\eta||m)\] 
    \item
    In the notations of item 2, if \(f\) is strictly convex and
    \(D_f(\pi_*\eta||\pi_*m)= D_f(\eta||m)<\infty\).
    Then \(\frac{d\eta}{dm}\) is \(Y\)-measurable, moreover, \(\frac{d\eta}{dm}=\frac{d\pi_*\eta}{d\pi_*m}\circ\pi\).
\end{enumerate}
\end{lemma}

\begin{definition}
Let \(X\) be a QI \(G\)-space.
\begin{itemize}
    \item 
    We denote by \(M(X)\) the set of measures on \(X\) in the given measure class on it.
    \item Given a probability measure \(\lambda\) on \(G\) we define
    the \textbf{\emph{Furstenberg \(\lambda,f\)-entropy}} to be the function \(h_{\lambda,f}:M(X)\to[0,\infty]\) given by:
    \[h_{\lambda,f}(\nu)=\sum_{g\in G} \lambda(g)D_{f}(g\nu||\nu)\]
    \item
    Define \(M_{\lambda,f}(X)=h_{\lambda,f}^{-1}[0,\infty)\).
\end{itemize}

\end{definition}

\begin{example}
For \(f(t)=t\ln(t)\) the  divergence \(D_{f}=D_{KL}\) is the famous KL-divergence.  
Notice that \(h_{\lambda,f}=h_{\lambda}\) is the Furstenberg entropy:
\[h_{\lambda}(\nu)=\sum_{g\in G}\lambda(g) \intop_{X} -\ln(\frac{dg^{-1}\nu}{d\nu})\:d\nu\]
\end{example}
An immediate application of Lemma \ref{Lemma f-divergence} yields:

\begin{lemma}\label{Lemma f-entropy}
Let \(\lambda\) be a probability measure on \(G\).
\begin{enumerate}
    \item
    For any QI-\(G\)-space \(X\), the function function \(h_{\lambda,f}\) is convex.
    \item 
    Let \(\pi:(X,\nu)\to (Y,m)\) be a factor of QI \(G\)-spaces. Then \(h_{\lambda,f}(X,\nu)\geq h_{\lambda,f}(Y,m)\). If \(\pi\) is a measure preserving extension we have equality. 
    \item
    Suppose \(\lambda\) is generating and \(f\) is strictly convex. If \(\pi:(X,\nu)\to(Y,m)\) is a factor such that \(h_{\lambda,f}(X,\nu)=h_{\lambda,f}(Y,m)<\infty\) then \(\pi\) is a measure preserving extension.
\end{enumerate}
\end{lemma}

\begin{definition}\label{definition minimal entropy number}
Let \(X\) be a QI \(G\)-space. Given a convex function \(f\) with \(f(1)=0\) and a probability measure \(\lambda\) on \(G\). We define the \textbf{\emph{minimal entropy number}} by:
\[I_{\lambda,f}(X)=\inf_{\nu\in M(X)} h_{\lambda,f}(X,\nu)\]
\end{definition}

\begin{lemma}\label{Lemma convolution-entropy}
Let \((X,\nu)\) be a QI \(G\)-space, 
and \(\kappa\in M(G)\). Then \(h_{\lambda,f}(\kappa)\geq h_{\lambda,f}(\kappa\ast\nu)\).
\end{lemma}
\begin{proof} 
Consider \(G\times X\) with \(G\) acting only on the first coordinate. Then \((G\times X,\kappa\times\nu)\) is measure preserving extension of \((G,\kappa)\), and is an extension of \((X,\kappa\ast\nu)\) via the action map. It follows from item 2 of Lemma \ref{Lemma f-entropy} that: 
\[h_{\lambda,f}(G,\kappa)=h_{\lambda,f}(G\times X, \kappa\times \nu)\geq h_{\lambda,f}(X,\kappa\ast\nu)\]
\end{proof}
The lemma immediately implies that the action of a group on itself is largest in terms of entropy:
\begin{corollary}\label{Cor Group-Space-Entropy}
For any QI \(G\)-space \(X\) we have \(I_{\lambda,f}(G)\geq I_{\lambda,f}(X)\).
\end{corollary}

Let us show that for ergodic measure classes, a minimize for entropy is unique (if exists):
\begin{lemma}\label{Lemma uniqueness of entropy minimize}
Suppose \(f\) is a strictly convex function with \(f(1)=0\) and \(\lambda\) is a finitely supported generating probability measure on \(G\). If \(X\) is a QI \(G\)-space which is ergodic (invariant sets are null or co-null), then \(h_{\lambda,f}\) is a strictly convex function on \(M_{\lambda,f}(X)\).\\ In particular there is at most one measure \(\nu\in M_{\lambda,f}(X)\) for which \(h_{\lambda,f}(X,\nu)=I_{\lambda,f}(X)\).
\end{lemma}
\begin{proof}
Fix \(\omega\in M(X)\). Let \(\nu_0,\nu_1\in M_{\lambda,f}(X)\) and \(\nu_{t}=(1-t)\nu_{0}+t \nu_{1}\). Note that as \(F\) is convex:
\begin{multline*}
    h_{\lambda,f}(\nu_{t})=\sum_{g\in G}\lambda(g) \intop_{X} F(\frac{dg\nu_{t}}{d\omega},\frac{d\nu_{t}}{d\omega})\:d\omega\leq\\
    \sum_{g\in G}\lambda(g) \intop_{X} (1-t)F(\frac{dg\nu_{0}}{d\omega},\frac{d\nu_{0}}{d\omega})+ t F(\frac{dg\nu_{1}}{d\omega},\frac{d\nu_{1}}{d\omega})\: d\omega=
    (1-t)h_{\lambda,f}(\nu_{0})+t h_{\lambda,f}(\nu_{1})
\end{multline*}
If one has equality \(h_{\lambda,f}(\nu_{t})=(1-t)h_{\lambda,f}(\nu_{0})+t h_{\lambda,f}(\nu_{1})<\infty\), according to Lemma \ref{Lemma F is convex} for any \(g\in Supp(\lambda)\) we have a.e. \(\frac{dg\nu_{0}}{d\nu_{0}}=\frac{dg\nu_{1}}{d\nu_{1}}\). However, this yields that \(\frac{d\nu_{0}}{d\nu_{1}}\) is \(G\)-invariant (since \(\lambda\) is generating) and using ergodicity we deduce \(\nu_{0}=\nu_{1}\). Thus \(h_{\lambda,f}\) is strictly convex.
\end{proof}

The previous lemma implies that for infinite groups, the infimum defining \(I_{\lambda,f}(G)\) is not attained:
\begin{corollary}\label{corollary no minimizing on G}
Let \(f\) be a strictly convex function with \(f(1)=0\) and \(\lambda\) be a generating probability measure on \(G\). Suppose that \(\kappa\) is a probability measure on \(G\) with \(I_{\lambda,f}(G)=h_{\lambda,f}(\kappa)<\infty\). Then \(G\) is finite and \(\kappa\) is the normalized counting measure.
\end{corollary}
\begin{proof}
Since the right \(G\)-action is an automorphism of the (left) \(G\)-space \(G\), Lemma \ref{Lemma uniqueness of entropy minimize} implies that \(\kappa\) is right \(G\)-invariant, concluding the proof. 
\end{proof}

\subsection{A criterion for minimizers of Entropy}\label{subsection Entropy minimize}
Our goal is to find a criterion for a measure on a QI \(G\)-space \(X\) to have minimal entropy \(h_{\lambda,f}\). We stress again that we consider only measures in the measure class given on \(X\).\\
In this subsection we assume in addition that \(f\) is smooth. Recall we denoted \(F(x,y)=yf(\frac{x}{y})\).

\begin{prop}\label{Prop f-divergence inequality}
Let \((X,\nu)\) be a BQI \(G\)-space. Then for any \(m\in M(X)\) and \(g\in G\) we have:
\[D_{f}(gm||m)\geq D_f(g\nu||\nu)+\intop_{X} \frac{\partial F}{\partial x} (1,\frac{dg^{-1}\nu}{d\nu})+\frac{\partial F}{\partial y}(\frac{dg\nu}{d\nu},1) \:\: d(m-\nu)\]
\end{prop}
\begin{proof}
Using the convexity of \(F\) we get the inequality 
\(F(v)\geq F(u)+ \nabla F(u)\cdot (v-u)
\) for any \(v,u\in (0,\infty)^{2}\). Fix \(\omega\in M(X)\) (e.g. \(\omega=\nu\)),
Substituting \(v=(\frac{dgm}{d\omega},\frac{dm}{d\omega}),u=(\frac{dg\nu}{d\omega},\frac{d\nu}{d\omega})\) we have a.e:
\begin{multline*}
    F(\frac{dgm}{d\omega},\frac{dm}{d\omega})\geq F(\frac{dg\nu}{d\omega},\frac{d\nu}{d\omega})+\big\langle\nabla F(\frac{dg\nu}{d\omega},\frac{d\nu}{d\omega}), (\frac{dgm}{d\omega},\frac{dm}{d\omega})-(\frac{dg\nu}{d\omega},\frac{d\nu}{d\omega})\big\rangle=\\
    F(\frac{dg\nu}{d\omega},\frac{d\nu}{d\omega})+\frac{\partial F}{\partial x}(\frac{dg\nu}{d\omega},\frac{d\nu}{d\omega})(\frac{dg(m-\nu)}{d\omega})+\frac{\partial F}{\partial y}(\frac{dg\nu}{d\omega},\frac{d\nu}{d\omega})(\frac{d(m-\nu)}{d\omega})
\end{multline*} 
As both partial derivatives of \(F\) are homogeneous of degree 0 and \(\nu\) is BQI, all terms in the above inequality are \(\omega\)-integrable. We conclude:
\[D_{f}(gm||m)\geq D_{f}(g\nu||\nu)+\intop_{X}\frac{\partial F}{\partial x}(\frac{dg\nu}{d\omega},\frac{d\nu}{d\omega})\:\:dg(m-\nu)+\intop_{X}\frac{\partial F}{\partial y}(\frac{dg\nu}{d\omega},\frac{d\nu}{d\omega})\:\:d(m-\nu)\]
Note that for \(p\in X\) we have \(\frac{dg\nu}{d\nu}(g\cdot p)=\frac{d\nu}{dg^{-1}\nu}(p)\) and thus:
\begin{multline*}
\intop_{X}\frac{\partial F}{\partial x}(\frac{dg\nu}{d\omega},\frac{d\nu}{d\omega})\:\:dg(m-\nu)+\intop_{X}\frac{\partial F}{\partial y}(\frac{dg\nu}{d\omega},\frac{d\nu}{d\omega})\:\:d(m-\nu)=\\
\intop_{X}\frac{\partial F}{\partial x}(\frac{dg\nu}{d\nu}(p),1)\:\:dg(m-\nu)(p)+\intop_{X}\frac{\partial F}{\partial y}(\frac{dg\nu}{d\nu}(p),1)\:\:d(m-\nu)(p)=\\
\intop_{X}\frac{\partial F}{\partial x}(\frac{dg\nu}{d\nu}(g\cdot p),1)\:\:d(m-\nu)(p)+
\intop_{X}\frac{\partial F}{\partial y}(\frac{dg\nu}{d\nu},1)\:\:d(m-\nu)=\\
\intop_{X} \big[\frac{\partial F}{\partial x}(1,\frac{dg^{-1}\nu}{d\nu})+\frac{\partial F}{\partial y}(\frac{dg\nu}{d\nu},1)\big] d(m-\nu)
\end{multline*}
This proves the desired inequality.
\end{proof}

\begin{prop}\label{Proposition f-entropy minimize}
Let \((X,\nu)\) be a BQI \(G\)-space and \(\lambda\) is a  finitely supported measure on \(G\). \\
Then \(h_{\lambda,f}(\nu)=I_{\lambda,f}(X)\) iff the function on \(X\)
\[\Psi_{\lambda,f}(\nu;\boldsymbol{\cdot}):=\sum_{g\in G} \lambda(g)\bigg(\frac{\partial F}{\partial x} (1,\frac{dg^{-1}\nu}{d\nu})+\frac{\partial F}{\partial y}(\frac{dg\nu}{d\nu},1)\bigg)\]
is constant (a.e.).
\end{prop}
\begin{proof}
Take \(m\in M(X)\). Using Proposition \ref{Prop f-divergence inequality} we get:
\[D_{f}(gm||m)\geq D_f(g\nu||\nu)+\intop_{X} \frac{\partial F}{\partial x} (1,\frac{dg^{-1}\nu}{d\nu})+\frac{\partial F}{\partial y}(\frac{dg\nu}{d\nu},1) \:\: d(m-\nu)\]
averaging over \(\lambda\) we conclude:
\[h_{\lambda,f}(m)\geq h_{\lambda,f}(\nu)+\intop_{X} \Psi_{\lambda,f}(\nu;x)\:\: d(m-\nu)(x)\]
If \(\Psi_{\lambda,f}(\nu;\cdot)\) is constant, then we get \(h_{\lambda,f}(m)\geq h_{\lambda,f}(\nu)\) and thus \(h_{\lambda,f}(\nu)=I_{\lambda,f}(X)\).\\
To finish the proof, we return to the other direction of the implication in the first statement. Suppose \(\nu\) is BQI measure such that \(h_{\lambda,f}(\nu)=I_{\lambda,f}(X)\). We will show \(\Psi_{\lambda,f}(\nu,\cdot)\) is constant. Consider any measure \(m\) in the bounded measure class of \(\nu\) and define \(\varphi(t)=h_{\lambda,f}((1-t)\nu+tm)\). We claim that:
\[\varphi^\prime(0)=\intop_{X} \Psi_{\lambda,f}(\nu;x) d(m-\nu)(x)\quad\quad(*)\]
This would finish the proof. Indeed, the minimality of \(\nu\) implies that the above integral is non-negative. As we can take any measure \(m\) in the bounded measure class, it follows easily that \(\Psi_{\lambda,f}(\nu;\cdot)\) must be constant a.e.\\
To verify \((*)\) we observe that by definition
\[\varphi(t)=\sum_{g}\lambda(g) \intop_{X} F\Big((1-t)\frac{dg\nu}{d\nu}+t\frac{dgm}{d\nu}\:,\:(1-t)+t\frac{dm}{d\nu}\Big)\:d\nu \]
Since \(\lambda\) has finite support and \(m,g\nu\) are in the bounded measure class of \(\nu\), we can differentiate under the integral sign and \((*)\) follows from a direct computation (as in the proof of Proposition \ref{Prop f-divergence inequality}).
\end{proof}

\begin{example}
For \(\lambda\) symmetric and finitely supported, we have a simpler formula for \(\Psi_{\lambda,f}\).\\
Denote \(\Psi_f(z)=\frac{\partial F}{\partial x}(1,z)+\frac{\partial F}{\partial y}(z,1)=f^{\prime}(\frac{1}{z})+f(z)-z f^\prime(z)\). Then we have:
\[\Psi_{\lambda,f}(\nu;\cdot)=\sum_{g\in G} \lambda(g)\Psi_{f}(\frac{dg\nu}{d\nu})\]
\end{example}

The next inequality will play a key role in our proof of Theorem \ref{Theorem fixed point for stationary actions}:
\begin{corollary}\label{Cor KL-entropy inequality}
Let \(X\) be a QI \(G\)-space, \(\mu\) a probability measure on \(G\) with finite support, and \(\nu\in M(X)\) a BQI measure. Then for any \(m\in M(X)\):
\[h_{\mu}(m)\geq \intop_{X} \big(1-\frac{d\mu\ast\nu}{d\nu}-\sum_{g\in G} \mu(g)\ln(\frac{dg^{-1}\nu}{d\nu})\big)\:\:dm\]
\end{corollary}
\begin{proof}
In this case we have \(F(x,y)=x\ln(\frac{x}{y})\) so that \(\frac{\partial F}{\partial x}(1,z)=-\ln(z)+1,\frac{\partial F}{\partial y}(z,1)=-z\). By Proposition \ref{Prop f-divergence inequality}:
\begin{multline*}
D_{KL}(gm||m)\geq D_{KL}(g\nu||\nu)+\intop_{X} \frac{\partial F}{\partial x} (1,\frac{dg^{-1}\nu}{d\nu})+\frac{\partial F}{\partial y}(\frac{dg\nu}{d\nu},1) \:\: d(m-\nu)=\\
=D_{KL}(g^{-1}\nu||\nu)+\intop_{X}\big[ -\ln(\frac{dg^{-1}\nu}{d\nu})+1-\frac{dg\nu}{d\nu}\big] \:\: d(m-\nu)=\intop_{X} \big[1-\frac{dg\nu}{d\nu}-\ln(\frac{dg^{-1}\nu}{d\nu})\big] \: dm
\end{multline*}
Averaging over \(\mu\), we obtain the desired inequality.
\end{proof}

\section{Ultralimits}\label{section Ultralimit}
In this section, we introduce an ultralimit for \(G\)-spaces via ultralimits of \(C^*\)-algebras. We begin by recalling the basic definition of ultralimit of Banach spaces and \(C^*\)-algebras. \\
\(\)\\
For this section, \(\mathcal{I}\) is a set and \(\mathcal{U}\) is an Ultra-filter on \(\mathcal{I}\). Recall that this means \(\mathcal{U}\subset 2^{\mathcal{I}}\) a collection of subsets of \(\mathcal{I}\), satisfying:\\
1. \(I\in \mathcal{U}, \emptyset\notin \mathcal{U}\). \\
2. \(\mathcal{U}\ni A\subset B\implies B\in\mathcal{U}\).\\
3. \(A,B\in\mathcal{U}\implies A\cap B\in\mathcal{U}\).\\
4. for any \(A\subset I\), either \(A\in\mathcal{U}\) or \(I\setminus A\in\mathcal{U}\).

\subsection{Ultralimit of a sequence}\label{subsection Ultralimit of a sequence}
We recall the standard definition of ultralimit of a sequence
\begin{definition}
Let \(\mathcal{I}\) be a set and \(\mathcal{U}\) be a Ultra-filter on \(\mathcal{I}\). 
\begin{itemize}
\item 
We say that a property \(P\) holds for \(\mathcal{U}\)-a.e. \(i \) iff 
\(\{i\in\mathcal{I}\big|P(i) \textrm{ is true} \}\in\mathcal{U} \)
\item
Given a compact Hausdorff space \( X \) and an \(\mathcal{I}\)-sequence \( (a_i)_{i\in\mathcal{I}}\), an element \(a\in X\) is called the \textbf{\emph{\(\mathcal{U}\)-ultralimit}} of \((a_i)\) iff 
for every open neighborhood \(W\) of \(a\), we have \(a_i\in W\) for \(\mathcal{U}\)-a.e. \(i\).\\
We denote \(a=\mathcal{U}\lim_{i\in \mathcal{I}} a_i \).
\end{itemize}
\end{definition}
We have the following well known:
\begin{lemma}\label{lemma ultralimit of sequences}
Let \(X,Y\) be compact Hausdorff spaces. Let \( (a_i)_{i\in \mathcal{I}},(b_i)_{i\in \mathcal{I}}\) be sequences in \(X\). Then:
\begin{enumerate}
    \item
    An ultralimit of \((a_i)\) exists and is unique.
    \item 
    If \(a_i=b_i\) for \(\mathcal{U}\)-a.e. \(i \), then \(\mathcal{U}\lim{a_i}=\mathcal{U}\lim{b_i}\).
    \item
    If \(f:X\rightarrow Y \) is a continuous function then \(f(\mathcal{U}\lim{a_i})=\mathcal{U}\lim{f(a_i)}\).
\end{enumerate}
\end{lemma}
\begin{example}
Considering the compact Hausdorff space \(\overline{\mathbb{R}}=\mathbb{R}\cup\{\pm \infty\} \) we have a well defined \(\mathcal{U}\)-ultralimit of any sequence of real numbers, and we also can define the \(\mathcal{U}\)-ultralimit of any bounded sequence of complex numbers. This ultralimit is a complex number.\\
By Lemma \ref{lemma ultralimit of sequences}, the resulting map \(\mathcal{U}\lim_{i\in\mathcal{I}}:\ell^\infty (\mathcal{I})\rightarrow\mathbb{C} \) is a \(C^*\)-homomorphism.
\end{example}

\subsection{Ultralimits of Banach spaces and C*-algebras}\label{subsection Ultralimit of Banach spaces and C* algebras}
We recall the basic notion from \cite{Heinrich1980UltraproductsIB}.
\begin{notation}
Let \(B_i\) be normed linear spaces (over \(\mathbb{C}\)) for \(i\in\mathcal{I}\), and let \(\mathcal{U}\) be an ultra-filter on \(\mathcal{I}\).\\
Define the following normed space:
\[V:=\ell^\infty(\mathcal{I},\{B_i\}):=\bigg\{(v_i)_{i\in\mathcal{I}}\in \prod_{i\in\mathcal{I}} B_i\Bigg|\:\sup_i ||v_i||_{B_i}<\infty\bigg\}\:,\quad ||(v_i)||_V=\sup_i ||v_i||_{B_i}\]
Consider the subspace: 
\(\mathcal{N}_\mathcal{U}=\bigg\{(v_i)\in V \Big|\: \mathcal{U}\lim ||v_i||_{B_i}=0\bigg\}\subset V\).
\end{notation}
One has the following formula for the quotient norm:
\begin{lemma}\label{lemma norm on ultralimit}
In the notations above:
\begin{itemize}
    \item \(\mathcal{N}_{\mathcal{U}}\subset V\) is a closed subspace.
    \item For all \(v=(v_i)\in V\)
\[\inf_{u\in\mathcal{N}_{\mathcal{U}}} ||v+u||_V=\mathcal{U}\lim ||v_i||_{B_i}\]
\end{itemize}
\end{lemma}

\begin{definition}
Define the \textbf{\emph{\(\mathcal{U}\)-ultralimit}} of the normed spaces \(B_i\) to be the normed space 
\[B:=\mathcal{U}\lim B_i:=\dfrac{\ell^\infty(\mathcal{I},\{B_i\})}{\mathcal{N}_{\mathcal{U}}}\]
We will write \(\mathcal{U}\lim v_i\) for the class of \((v_i)\) in \(B\). By Lemma \ref{lemma norm on ultralimit} we have \(||\mathcal{U}\lim v_i||_B=\mathcal{U}\lim ||v_i||_{B_i}\).
\end{definition}

\begin{lemma}\label{Lemma Ulim is banach}
Suppose \((B_i)_{i\in\mathcal{I}}\) are Banach spaces, then:
\begin{enumerate}
\item
\(\mathcal{U}\lim B_i\) is a Banach space.
\item
Let \(a^{(n)}_i\in B_i\), and suppose that \(M_n\) are positive numbers with \(||a^{(n)}_i||_{B_i}\leq M_n\) and \(\sum M_n<\infty\). Then we have: 
\[\mathcal{U}\lim_i \sum_n a^{(n)}_i=\sum_n \mathcal{U}\lim_i a^{(n)}_i\]
\end{enumerate}
\end{lemma}
\begin{proof}
For (1): the quotient of a Banach space by a closed subspace is a Banach space.\\
For (2): We have equality \((\sum_{n} a_i^{(n)})_{i}=\sum_{n} (a_i^{(n)})_i\) on \(V\) and the projection to \(B\) is continuous.   
\end{proof}

To define the ultralimit of \(C^{*}\) algebras, we make the following observation:

\begin{lemma}\label{Lemma Ulim is C algebra}
Suppose \((B_i)_{i\in\mathcal{I}}\) are \(C^*\) algebras.
Then \(B=\mathcal{U}\lim B_i\) has a unique structure of \(C^*\) algebra satisfying \((\mathcal{U}\lim a_i)^*=\mathcal{U}\lim a_i^*,\: (\mathcal{U}\lim a_i)\cdot(\mathcal{U}\lim b_i)=\mathcal{U}\lim a_i\cdot b_i\).\\
If \(B_i\) are commutative then so is \(B\). If \(B_i\) all have unit then \(\mathcal{U}\lim 1_{B_i}\) is a unit for \(B\).
\end{lemma}
\begin{proof}
Indeed, \(V=\ell^{\infty}(\mathcal{I},\{B_i\})\) is a \(C^*\) algebra with the operations \((a_i)^*=(a_i^*),\: (a_i)\cdot(b_i)=(a_i\cdot b_i)\). Clearly, \(\mathcal{N}_{\mathcal{U}}\) defined above is a \(^*\)-two sided ideal. So the quotient is a \(C^*\)-algebra, the lemma follows.
\end{proof}

\begin{definition}
Given \(C^*\) algebras \(B_i\), their ultralimit \(\mathcal{U}\lim B_i\) is the ultralimit as above together with the \(C^*\)-algebra structure as in Lemma \ref{Lemma Ulim is C algebra}.
\end{definition}

The following basic lemma is useful:
\begin{lemma}\label{Lemma continous calculus and ultralimit}
Let \(A_i\) be \(C^*\) algebras, \(A=\mathcal{U}\lim A_i\) be the ultralimit.
\begin{enumerate}
    \item 
    If \(a_i\in A_i,\: \sup ||a_i||<\infty\) are self adjoint/positive/normal then \(a=\mathcal{U}\lim a_i\) is self adjoint/positive/normal. 
    \item 
    If \(a_i\in A\) are normal and \(||a_i||\leq M\), then for any continuous function \(f\) on the disc \(\overline{D}(0,M)\) we have \(f(\mathcal{U}\lim a_i)=\mathcal{U}\lim f(a_i)\).
    \item
    If \(a\in A\) is self adjoint/positive then we can find \(a_i\in A_i\) self adjoint/positive with \(||a_i||=||a||, a=\mathcal{U}\lim a_i\).

\end{enumerate}
\end{lemma}

Finally we would like to discuss duality: for a normed space \(B\) we denote the dual normed space by \(B^*\) and for \(v\in B,f\in B^*\) we denote \(\langle v,f\rangle=f(v)\).\\
It is easy to see one has a canonical isometric embedding \(\Phi: \mathcal{U}\lim_i B^*_i\hookrightarrow (\mathcal{U}\lim B_i)^*\) described by the formula
\(\langle\mathcal{U}\lim v_i, \Phi(\mathcal{U}\lim f_i)\rangle=\mathcal{U}\lim \langle v_i,f_i\rangle\). Hereafter we shall omit \(\Phi\) from the notations.

\subsection{Ultralimits of C*-probability spaces and the G-equivariant case}\label{subsection Ultralimit of C*-probability spaces and the G-equivariant case}

We shall need to define the ultralimit of \(C^*\)-probability spaces (see Definition \ref{definition C probability space}):

\begin{definition}
Let \((A_i,\nu_i)_{i\in\mathcal{I}}\) be a collection of \(C^*\)-probability spaces. Define their ultralimit: 
\[(A,\nu)=\mathcal{U}\lim(A_i,\nu_i):=\big(\mathcal{U}\lim A_i,\mathcal{U}\lim \nu_i)_{red}\]
\end{definition}
Given \((a_i)\in\ell^{\infty}(\mathcal{I}, \{A_i\})\) we will denote by \(\mathcal{U}\lim a_i\) the element of \(A\) which is the class of  \(\mathcal{U}\lim a_i\in \mathcal{U}\lim A_i\) in \(A=\frac{\mathcal{U}\lim A_i}{\mathcal{N}_{\mathcal{U}\lim\nu_i}}\).

\begin{lemma}\label{Lemma propetries of ultralimit of C prob spaces}
Let \((A_i,\nu_i)\) be \(C^*\)-probability spaces, \((A,\nu)=\mathcal{U}\lim (A_i,\nu_i)\).
\begin{enumerate}
    \item 
If \(a_i\in A_i, ||a_i||\leq M, f\in C(\overline{D}(0,M))\) then:
\begin{itemize}
    \item
    \(f(\mathcal{U}\lim a_i)=\mathcal{U}\lim f(a_i)\)
    \item
    \(\nu(\mathcal{U}\lim a_i)=\mathcal{U}\lim \nu_i(a_i)\)
\end{itemize}
\item
If \(\eta_i\) are states on \(A_i\) with \(\eta_i\ll^b\nu_i\) and \(\sup ||\frac{d\eta_i}{d\nu_i}||\leq M\), then on \(\mathcal{U}\lim A_i\) we have \(\mathcal{U}\lim \eta_i \ll^b \mathcal{U}\lim \nu_i\). In particular \(\mathcal{U}\lim \eta_i\) induces a probability measure \(\eta\) on \(A\) and \(\frac{d\eta}{d\nu}=\mathcal{U}\lim \frac{d\eta_i}{d\nu_i}\)
\end{enumerate}
\end{lemma}

\begin{proof}
\begin{enumerate}
    \item Obvious from Lemma \ref{Lemma continous calculus and ultralimit}.
    \item Obvious since for \(a=\mathcal{U}\lim a_i\in\mathcal{U}\lim A_i\) we have: 
\[(\mathcal{U}\lim\nu_i)(a\cdot\: \mathcal{U}\lim\frac{d\eta_i}{d\nu_i})=\mathcal{U}\lim \nu_i\big(a_i\cdot\:\frac{d\eta_i}{d\nu_i}\big)=\mathcal{U}\lim \eta_i(a_i)=(\mathcal{U}\lim \eta_i)(\mathcal{U}\lim a_i)\]
\end{enumerate}
\end{proof}

\begin{lemma}\label{Lemma functoriality C spaces}
One has the following:
\begin{enumerate}
    \item Let \((A_i,\nu_i)\) and \((B_i,m_i)\) be \(C^*\)-probability spaces and \(T_i:(B_i,m_i)\to(A_i,\nu_i)\) be factors. Then there is a factor map \(T=\mathcal{U}\lim T_i:\mathcal{U}\lim (B_i,m_i)\to \mathcal{U}\lim (A_i,\nu_i)\) satisfying \(T(\mathcal{U}\lim b_i)=\mathcal{U}\lim T_i(b_i)\).
    \item
    For any \(C^*\)-probability space \((A,\nu)\) the diagonal mapping \(\Delta(a)=\mathcal{U}\lim_i a\) defines a factor \(\Delta:(A,\nu)\to\mathcal{U}\lim_i (A,\nu)\).
\end{enumerate}
\end{lemma}
\begin{proof}
\begin{enumerate}
    \item
    The mapping \(T:\ell^\infty(\mathcal{I},{B_i})\to \mathcal{U}\lim (A_i,\nu_i)\) given by \(T((b_i))=\mathcal{U}\lim T_i(b_i)\) is a \(^*\)-homomorphism. As  
    \(T^*(\mathcal{U}\lim \nu_i)((b_i))=(\mathcal{U}\lim \nu_i)(\mathcal{U}\lim T_i(b_i))=p^*(\mathcal{U}\lim m_i)((b_i))\) 
    where \(p\) is the projection \(p:\ell^\infty(\mathcal{I},{B_i})\to \mathcal{U}\lim (B_i,m_i)\), we conclude that \(T\) gives rise to a factor map \(T:\mathcal{U}\lim (B_i,m_i)\to \mathcal{U}\lim (A_i,\nu_i)\) with the desired property.
    \item
    \(\Delta\) is an \(*\)-homomorphism \(\Delta: A\to \frac{\mathcal{U}\lim A}{\mathcal{N}_\nu}\) and \(\Delta^*\big(\mathcal{U}\lim_i\nu\big)(a)=\mathcal{U}\lim_i \nu(a)=\nu(a)\).
\end{enumerate}
\end{proof}
Let \(G\) be a discrete countable group. \\
Consider a collection \(\big(\mathcal{A}_i\big)_{i\in\mathcal{I}}\) of BQI \(G\)-\(C^*\)-spaces. We will say that \(\mathcal{A}_i\) are uniformly BQI, if there is a function \(M:G\to\mathbb{R}_{+}\) which is a majorant for all of them (Recall the definition of a majorant from Definition \ref{definition BQI G C space}).\\
Given such a collection, we next define its ultralimit: 
\begin{definition}
Let \(\big(\mathcal{A}_i\big)_{i\in\mathcal{I}}\) be uniformly BQI \(G\)-\(C^*\)-spaces. Consider \(\mathcal{A}=\mathcal{U}\lim \mathcal{A}_i\). By Functoriality (Lemma \ref{Lemma functoriality C spaces}), we get a natural \(G\)-\(C^*\)-space structure on \(\mathcal{A}\), such that: when \(a_i\in A_i\) are uniformly bounded then \({}^g\mathcal{U}\lim a_i=\mathcal{U}\lim {}^ga_i\).  We will say that \(\mathcal{A}\) is the \emph{ultra-limit} of the uniformly BQI \(G\)-\(C^*\)-spaces \(\mathcal{A}_i\).
\end{definition}

\begin{lemma}\label{Lemma RN cocyle in C ultralimit}
Let \(\big(\mathcal{A}_i\big)_{i\in\mathcal{I}}\) be uniformly  \(C^*\)-BQI-\(G\)-spaces, and let \(\mathcal{A}=(A,\nu)=\mathcal{U}\lim (A_i,\nu_i)\).\\
Let \(\mu\) be a probability measure on \(G\) and \(a_i\in A_i\) be uniformly bounded elements. Then:
\begin{enumerate}
    \item For any \(g\in G\) we have: \(R_{\mathcal{A}}(g)=\mathcal{U}\lim R_{\mathcal{A}_i}(g)\).
    \item
    \(\mu\ast\mathcal{U}\lim a_i=\mathcal{U}\lim \mu\ast a_i\)
    \item
    \((\mu\ast\nu)(\mathcal{U}\lim a_i)=\mathcal{U}\lim (\mu\ast\nu_i)(a_i)\)
\end{enumerate}
\end{lemma}
\begin{proof}
Item (1) is immediate from Lemma \ref{Lemma propetries of ultralimit of C prob spaces}, whereas items (2) and (3) are immediate from Lemma \ref{Lemma Ulim is banach} since \(||\mu(g) {}^ga_i||\leq \mu(g)(\sup_{i}||a_i||)\:,\: ||\mu(g)\cdot {}^g\nu||\leq \mu(g)\) and \(\sum_g \mu(g)\) is a convergent sum.
\end{proof}

\begin{lemma}\label{Lemma functoriality G C spaces}
We have the following:
\begin{enumerate}
    \item Let \((A_i,\nu_i)\) and \((B_i,m_i)\) be uniformly BQI \(G\)-\(C^*\)-spaces and \(T_i:(B_i,m_i)\to(A_i,\nu_i)\) be factors. Then there is a factor map \(T=\mathcal{U}\lim T_i:\mathcal{U}\lim (B_i,m_i)\to \mathcal{U}\lim (A_i,\nu_i)\) satisfying \(T(\mathcal{U}\lim b_i)=\mathcal{U}\lim T_i(b_i)\).\\
    If each \(T_i\) is measure preserving extension then so is \(T\).
    \item
    For any BQI \(G\)-\(C^*\)-space \((A,\nu)\), the diagonal mapping \(\Delta(a)=\mathcal{U}\lim_i a\) defines a measure preserving extension \(\Delta:(A,\nu)\to\mathcal{U}\lim_i (A,\nu)\).
\end{enumerate}
\end{lemma}
\begin{proof}
Note that \(T,\Delta\) from Lemma \ref{Lemma functoriality C spaces} are \(G\)-equivariant. Also if all \(T_i\) are measure preserving extensions then \(T(R_{\mathcal{U}\lim m_i}(g))=T(\mathcal{U}\lim R_{m_i}(g))=\mathcal{U}\lim T_i(R_{m_i}(g))=\mathcal{U}\lim R_{\nu_i}(g)=R_{\nu}(g)\).\\
Since \(\Delta(R_{\nu}(g))=\mathcal{U}\lim R_{\nu}(g)=R_{\mathcal{U}\lim\nu}(g)\) we get that \(\Delta\) is a measure preserving extension. 
\end{proof}

\subsection{Ultralimits for G-spaces via algebras}\label{subsection Ultralimit for G-spaces via algebras}

\begin{definition}\label{Definition ultralimit of G measure spaces}
Let \(\mathcal{X}_i=(X_i,\nu_i)\:\: (i\in\mathcal{I})\) be a collection of uniformly BQI \(G\)-spaces, in the sense that there is a uniform majorant for the whole family. 
Consider \(\mathcal{A}_i=(L^\infty(X_i),
\nu_i)\), which are uniformly BQI \(G\)-\(C^*\)-algebras, and let \(\mathcal{A}=\mathcal{U}\lim \mathcal{A}_i=(A,\nu)\). The Gelfand space \(X\) for \(A\) is a BQI \(G\)-space that will be denoted by \(\mathcal{U}\lim_{big} \mathcal{X}_i\). We call it a version of the ultralimit of the actions \(\mathcal{X}_i\).\\
The RN factor of \(\mathcal{A}\) will be denoted by \(\mathcal{U}\lim_{RN} \mathcal{X}_i\), we consider it is using the canonical topological RN model of Proposition \ref{Proposition RN model for RN factor}, that is, \(C(\mathcal{U}\lim_{RN} \mathcal{X}_i)=\mathcal{A}_{RN}\).
\end{definition} 
Note that if \(\varphi_i\) are uniformly bounded functions on \(X_i\), then we have a well defined function \(\varphi\in L^\infty(X)\) which corresponds to \(\mathcal{U}\lim \varphi_i\in A\) (for which we keep this notation). \\
Lemma \ref{Lemma RN cocyle in C ultralimit} implies that \[\mathcal{U}\lim\frac{dg\nu_i}{d\nu_i}=\frac{dg\nu}{d\nu}\]
The basic result we shall need is that entropy behaves well under ultralimits:
\begin{lemma}\label{Lemma convergence of RN parameters}
Let \((X_i,\nu_i)\) be uniformly BQI \(G\)-spaces and \((X,\nu)=\mathcal{U}\lim_{big} (X_i,\nu_i)\).\\
If \(f\in C(0,\infty)\) is a continuous function, and \(g\in G\) then
\[    \mathcal{U}\lim \intop_{X_i}f(\frac{dg\nu_i}{d\nu_i}) \:d\nu_i= \intop_{X}f(\frac{dg\nu}{d\nu}) \:d\nu\]
\end{lemma}
\begin{proof}
Let \(M\) be a uniform majorant for \(\mathcal{X}_i\), then \(e^{-M(g)}\leq \frac{dg\nu_i}{d\nu_i}\leq e^{M(g)}\). From Lemma \ref{Lemma propetries of ultralimit of C prob spaces} we get:
\begin{multline*}
    \mathcal{U}\lim \intop_{X_i}f(\frac{dg\nu_i}{d\nu_i}) \:d\nu_i=\mathcal{U}\lim\: \nu_i\big(f(\frac{dg\nu_i}{d\nu_i})\big)=\nu\Big(\mathcal{U}\lim f(\frac{dg\nu_i}{d\nu_i})\Big)=\\
    \nu\Big( f\big(\mathcal{U}\lim\frac{dg\nu_i}{d\nu_i}\big)\Big)=
    \nu\Big( f\big(\frac{dg\nu}{d\nu}\big)\Big)=\intop_{X}f(\frac{dg\nu}{d\nu}) \:d\nu
\end{multline*}
\end{proof}
\begin{corollary}\label{corollary convergence of entropy}
Let \((X_i,\nu_i)\) be uniformly BQI \(G\)-spaces and \((X,\nu)=\mathcal{U}\lim_{big} (X_i,\nu_i)\).
Let \(f\) be a convex function with \(f(1)=0\), \(\lambda\) a finitely supported probability measure on \(G\). Then we have:
\[\mathcal{U}\lim h_{\lambda,f}(X_i,\nu_i)= h_{\lambda,f}(X,\nu)\]
\end{corollary}
\begin{proof}
Follows from Lemma \ref{Lemma convergence of RN parameters} by linearity.
\end{proof}

\begin{remark}
As remarked in the introduction, in the sequel \cite{ultrAmenable} the theory of ultralimit of uniformly QI \(G\)-spaces will be developed from the measure theory point of view, and this is the reason for \enquote{big} in the notation \(\mathcal{U}\lim_{big} \mathcal{X}_i\).
\end{remark}
\begin{remark}
Using the continuity of \(\mathcal{U}\lim\) for sequences, one can deduce that the equality in Corollary \ref{corollary convergence of entropy} holds for more general \(\lambda\)'s by controlling the tail term. A concrete example is given in Remark \ref{corollary entropy KV using ultralimits}.
\end{remark}

\begin{remark}\label{remark ultralimits are via algebras}
In the paper \cite{conley_kechris_tucker-drob_2013} the authors considered a construction of ultralimit of measure preserving actions using a construction of ultralimit of probability spaces due to Loeb (for the case of finite probability spaces see \cite{elek2007limits}). Our usage of \(C^{*}\)-algebras is because it was the fastest and least technical way to be able to state and prove Theorem \ref{Theorem E}. However, it limits us to work with bounded functions and in particular to restrict ourselves to work with bounded-quasi-invariant actions (see definition \ref{definition BQI G C space}) and to make the assumption of uniformly bounded-quasi-invariant (see definition \ref{Definition ultralimit of G measure spaces}). This suffices for our for our paper.\\
The sequel paper \cite{ultrAmenable} elaborates on the ultralimit construction, and in particular develops a measure theoretic ultralimit for quasi-invariant actions as well, under a natural assumption of \enquote{uniformly quasi-invariance} (which is more technical). Another goal is to show the equivalence between the construction developed here and the measure theoretic one taken there. \\
We believe that this method of ultralimits could be intriguing and useful in further research in the ergodic theory of group actions.
\end{remark}

\section{The Furstenberg-Poisson boundary as an ultralimit}\label{section Furstenberg-Poisson boundary as ultralimit}
In this section we provide a novel construction of the Furstenberg-Poisson boundary \(\mathcal{B}(G,\mu)\) of a discrete countable group \(G\) together with a generating measure \(\mu\). 
\begin{lemma}\label{Lemma almost stationary implies BQI}
Assume \(\mu\) is a generating probability on \(G\) and \(C>0\). Let \(\nu\) is a measure on a measurable \(G\)-space \(X\) such that \(\mu\ast\nu\leq C\cdot\nu\). Then \((X,\nu)\) is BQI with a majorant depending only on \(C,\mu\).
\end{lemma}
This is clear, as
\(\mu(g)\frac{dg\nu}{d\nu}\leq \frac{d(\mu\ast\nu)}{d\nu}\leq C\), hence \(\frac{dg\nu}{d\nu}\leq\frac{C}{\mu(g)}\).
\begin{notation}
Let \((G,\mu)\) be a measured discrete countable group, with \(\mu\) generating. For any \(0<a<1\), define the Abel sum 
\[\mu_a:=(1-a)\sum_{n=0}^\infty a^n\cdot\mu^{*n}\]
\end{notation}
We have the following:
\begin{lemma}\label{Lemma properties of abel measures}
With the notations above:
\begin{itemize}
    \item \(\mu_a\) is a probability measure on \(G\).
    \item We have the following formula \[\mu\ast\mu_a=(1-a)\sum_{n=0}^\infty a^n\cdot\mu^{\ast(n+1)}=\frac{\mu_a}{a}-\frac{1-a}{a}\delta_e\]
    \item For \(a>\frac{1}{2}\), the collection of measures \(\mu_a\) is uniformly BQI.
\end{itemize}
\end{lemma}
The first two items are obvious, the last item follows from the second one and Lemma \ref{Lemma almost stationary implies BQI}.\\
Now we give the main definition of this section:
\begin{definition}
Let \(\mathcal{U}\) be an ultra-filter on \((\frac{1}{2},1)\) which contains \((1-\epsilon,1)\) for all \(\epsilon>0\).\\ Define the \(\mathcal{U}\)-\emph{Abel \(C^*\) space} of \((G,\mu)\) to be the following ultralimit
        \[\mathcal{B}^{alg}_{\mathcal{U}}(G,\mu)=\mathcal{U}\lim_a (L^\infty(G,\mu_a),\mu_a)\]
together with the probability measure \(\nu_\mu=\mathcal{U}\lim \mu_a\) on \(\mathcal{B}^{alg}_{\mathcal{U}}(G,\mu)\).\\
Denote by \(\mathcal{B}_{\mathcal{U}}(G,\mu)_{RN}\) the compact metrizable BQI \(G\)-space corresponding to \(\mathcal{B}^{alg}_{\mathcal{U}}(G,\mu)_{RN}\), where the latter is the Radon-Nikodym factor of \(\mathcal{B}^{alg}_{\mathcal{U}}(G,\mu)\).
\end{definition}
Note that \(\mathcal{B}_{\mathcal{U}}(G,\mu)_{RN}\) is the canonical model of Proposition \ref{Proposition RN model for RN factor} for the RN factor of \(\mathcal{U}\lim_{big} (G,\mu_a)\).
 
\begin{prop}\label{Proposition Abel space is stationary}
We have \(\mu\ast\nu_{\mu}=\nu_{\mu}\). Thus \((\mathcal{B}^{alg}_{\mathcal{U}}(G,\mu),\nu_\mu)\), \(\mathcal{B}_{\mathcal{U}}(G,\mu)_{RN}\) are \(\mu\)-stationary.
\end{prop}
\begin{proof}
Using Lemma \ref{Lemma properties of abel measures} we have the following identity in \(\mathcal{U}\lim_a \:( L^\infty(G)^{*})\)
\[\mu\ast\nu_\mu=\mathcal{U}\lim_a \mu\ast\mu_a=\mathcal{U}\lim_a \frac{\mu_a}{a}-\frac{1-a}{a}\delta_e=\mathcal{U}\lim_a \mu_a=\nu_\mu\]
Thus we get this identity when we consider \(\nu_{\mu}\) as a state on the underlying \(C^{*}\)-algebra of \(\mathcal{B}_{\mathcal{U}}^{alg}(G,\mu)\).
\end{proof}

Our main goal for this section is to show that \(\mathcal{B}_{\mathcal{U}}(G,\mu)_{RN}\) is the RN model of the Furstenberg-Poisson boundary of \((G,\mu)\) of Corollary \ref{corollary RN model for Furstenberg-Poisson boundary}.

\begin{lemma}\label{Lemma diagram G times X}
\(\)
\begin{enumerate}
\item 
Given a QI \(G\)-space \((X,\nu)\) and a quasi-invariant measure \(\eta\) on \(G\), define \(Y=G\times X\) with the \(G\)-action \(g(h,x)=(gh,x)\) and the product measure \(\eta\times\nu\). Then \((Y,\eta\times\nu)\) is a QI G-space. The mappings \(m(g,x)=gx,p(g,x)=g\) define a diagram:
\[\begin{CD}
(Y,\eta\times\nu) @>m>>  (X,\eta\ast\nu) \\
@VVpV \\
(G,\eta)
\end{CD}
\]
in which \(m\) is a factor map and \(p\) is measure preserving extension.
\item
Dually, let \((A,\nu)\) be a BQI \(C^*-G\)-space, \(\eta\) a BQI measure on \(G\). Then we have a BQI \(C^*-G\)-space \((C,\tau)\) and a diagram:
\[\begin{CD}
(A,\eta\ast\nu) @>m>> (C,\tau) \\
@. @AApA\\
@. (L^\infty(G),\eta)
\end{CD}
\]
in which \(m\) is a factor and \(p\) is measure preserving extension.

\end{enumerate}
\end{lemma}

\begin{proof}
\begin{enumerate}
\item 
It is obvious that \(p\) is \(G\)-equivariant and that \(p_*(\eta\times\nu)=\eta\). Moreover, since the action is only on the first component we conclude that 
\(\frac{dg(\eta\times\nu)}{d(\eta\times\nu)}=\frac{dg\eta}{d\eta}\circ p\)
which implies that \(p\) is a measure preserving extension. \\
The equalities \(m\big(g\cdot(h,x)\big)=m(gh,x)=(g\cdot h)\cdot x=g\cdot(h\cdot x)=g\cdot m(h,x)\) and \(m_*(\eta\times\nu)=\eta\ast\nu\) implies that \(m\) is a factor.   
\item
By Gelfand's representation theorem, \(A\cong C(X)\) where \(X\) is a compact Hausdorff topological space, \(G\) acts on \(X\) continuously and \(\nu\) is a regular probability measure on \(X\). We have that \(\frac{dg\nu}{d\nu}\in C(X)\) and in particular bounded, hence \((X,\nu)\) is a BQI \(G\)-space. Applying the first item, take \((C,\tau)=(L^\infty(Y),\eta\times\nu)\) and \(m\) to be the composition \(A\cong C(X)\subset L^\infty(X,\nu)\longrightarrow (C,\tau)\). The lemma follows.
\end{enumerate}
\end{proof}

Let us show a universality property of \(\mathcal{B}_{\mathcal{U}}^{alg}(G,\mu)\):
\begin{prop}\label{Proposition universal property of abel}
Let \((A,\nu)\) be a \(C^*\)-BQI \(G\)-space which is \(\mu\)-stationary space. Then we have \((C,\tau)\) a \(C^*\)-BQI \(G\)-space which is \(\mu\)-stationary together with a diagram:
\[\begin{CD}
(A,\nu) @>m>> (C,\tau) \\
@. @AApA\\
\mathcal{B}^{alg}_{\mathcal{U}}(G,\mu)_{RN} @>i>>\mathcal{B}_{\mathcal{U}}^{alg}(G,\mu)
\end{CD}
\]
Where \(m\) is a factor, and \(p,i\) are measure preserving extensions.
\end{prop}
\begin{proof}
It is clear that \(i\) is a measure preserving extension.
Consider \(\eta=\mu_a\:(a>\frac{1}{2})\), note that \(\eta\ast\nu=\nu\) and thus by Lemma \ref{Lemma diagram G times X}  we have \((C_a,\tau_a)\) with
\[\begin{CD}
(A,\nu) @>m_a>> (C_a,\tau_a) \\
@. @AAp_aA\\
@. (L^\infty(G),\mu_a)
\end{CD}\]
Since \(\mu_a\) are uniformly BQI and \(p_a\) is measure preserving, we conclude that \(\tau_a\) are uniformly BQI.
We take \((C,\tau):=\mathcal{U}\lim (C_a,\tau_a)\). Using functionality (Lemma \ref{Lemma functoriality G C spaces}), the maps \(\mathcal{U}\lim_a m_a,\mathcal{U}\lim_a p_a\) fit into the following diagram:
\[
\begin{CD}
(A,\nu) @>\Delta>> \mathcal{U}\lim_a (A,\nu) @>\mathcal{U}\lim_a m_a>> (C,\tau)=\mathcal{U}\lim_a (C_a,\tau_a) \\
@. @. @AA\mathcal{U}\lim_a p_a A  \\
@. @.  \mathcal{B}^{alg}_{\mathcal{U}}(G,\mu)=\mathcal{U}\lim_a(L^\infty(G,\mu_a),\mu_a)
\end{CD}
\]
such that \(m:=\mathcal{U}\lim m_a\circ\Delta\) is a factor, \(p=\mathcal{U}\lim p_a\) is a measure preserving extension, in particular \((C,\tau)\) is \(\mu\)-stationary.

\end{proof}
We shall need the following well known result, whose proof follows by comparing the conditional measures (see \cite{Furstenberg1973BoundaryTA} or \cite{BaderShalom}) to the disintegration.
\begin{lemma}\label{Lemma Furstenberg Glasner universal property}
Let \(\mathcal{B}(G,\mu)\) be the Furstenberg-Poisson boundary of \((G,\mu)\). Let \(\mathfrak{X}\) be a \(\mu\) stationary \(G\)-space that is an extension of the Furstenberg-Poisson boundary \(p:\mathfrak{X}\to\mathcal{B}(G,\mu)\). Then \(p\) is a measure preserving extension.
\end{lemma}

We now come to our characterization of the Furstenberg-Poisson boundary, which implies Theorem \ref{Theorem E}:
\begin{theorem}\label{Theorem Abel space is Furstenberg-Poisson boundary}
Let \(G\) be a discrete countable group, and \(\mu\) a generating measure on \(G\). Then \(\mathcal{B}_{\mathcal{U}}(G,\mu)_{RN}\) is the canonical RN model of the Furstenberg-Poisson boundary of \((G,\mu)\).
\end{theorem}
\begin{remark}
Recall from Corollary \ref{corollary RN model for Furstenberg-Poisson boundary} that the Furstenberg-Poisson boundary admits a canonical RN model.
\end{remark}

\begin{proof}[Proof of the Theorem]
Consider \(\mathcal{B}(G,\mu)\) the canonical RN model to the Furstenberg-Poisson boundary, and consider \((C(\mathcal{B}(G,\mu)),\nu)=(L^\infty(\mathcal{B}(G,\mu)),\nu)_{RN}\). It is \(\mu\)-stationary, and thus by Proposition \ref{Proposition universal property of abel} there exists a \(\mu\)-stationary \(C^*-G\)-space \((C,\tau)\) and a diagram:
\[
\begin{CD}
C(\mathcal{B}(G,\mu)) @>m>> (C,\tau) \\
@. @AA p A  \\
\mathcal{B}^{alg}_{\mathcal{U}}(G,\mu)_{RN}  @>i>> \mathcal{B}^{alg}_{\mathcal{U}}(G,\mu)
\end{CD}
\]
Here \(q=p\circ i\) is a measure preserving extension. Thus:
\[(C(\mathcal{B}_{\mathcal{U}}(G,\mu)_{RN}),\nu_{\mu})\cong\mathcal{B}^{alg}_{\mathcal{U}}(G,\mu)_{RN} \cong (C,\tau)_{RN}\]
Let \(\mathcal{X}=(X,\tau)\) be a Gelfand space for \((C,\tau)\). It is a \(\mu\)-stationary space and \(m\) gives rise to a factor map \(m: \mathcal{X}\to \mathcal{B}(G,\mu)\). By Lemma \ref{Lemma Furstenberg Glasner universal property} we conclude that \(m\) is a measure-preserving extension, which implies that \(m\) induces an isomorphism \((C,\tau)_{RN}\cong (C(\mathcal{B}(G,\mu)),\nu)\).\\
In conclusion we get \((C(\mathcal{B}(G,\mu)),\nu)\cong (C(\mathcal{B}_{\mathcal{U}}(G,\mu)_{RN}),\nu_{\mu})\), and the proof is complete.
\end{proof}

\begin{corollary}\label{Corollary RN parameters on group and boundary}
Let \(G\) be a discrete countable group, \(\mu\) a generating probability measure on \(G\). Denote by \((\mathcal{B},\nu)\) the Furstenberg-Poisson boundary of \((G,\mu)\). Then for any \(f\in C(0,\infty)\) and \(g\in G\) we have:
\[\lim_{a\to1^{-}} \sum_{x\in G} f\bigg(\frac{\mu_a(g^{-1}x)}{\mu_a(x)}\bigg)\mu_a(x)=\intop_{\mathcal{B}} f\big(\frac{dg\nu}{d\nu}(x)\big)d\nu(x)\]
In particular, if \(f\) is convex with \(f(1)=0\) and \(\lambda\) is a finitely supported measure on \(G\) we have:
\[\lim_{a\to 1^{-}} h_{\lambda,f}(G,\mu_{a})=h_{\lambda,f}(\mathcal{B}(G,\mu),\nu_{\mu})\]
\end{corollary}
\begin{proof}
Follows immediately from Lemma \ref{Lemma convergence of RN parameters} and Theorem \ref{Theorem Abel space is Furstenberg-Poisson boundary} using the following observation: \(\lim_{a\to 1^{-}} R(a)=R\) iff for any ultrafilter \(\mathcal{U}\) on \((0,1)\) with \(\mathcal{U}\lim a=1\) one has \(\mathcal{U}\lim R(a)=R\). 
\end{proof}

\begin{remark}\label{remark martin boundary}
We note that it is possible to prove the above corollary without applying our ultralimit construction, in a more direct way by considering Martin boundary (see chapter 24 of \cite{woess2000random} and \cite{woess2021ratio} and noticing that the Green function is actually \(G(g,x|a) = \frac{g\mu_{a}(x)}{1-a}\)). It could be interesting to relate those constructions.
\end{remark}

One can deduce from Corollary \ref{Corollary RN parameters on group and boundary} the fundamental Theorem 3.1 in \cite{BoundaryEntropy}:

\begin{corollary}\label{corollary entropy KV using ultralimits} 
\[h_{\mu}(\mathfrak{B}(G,\mu),\nu_{\mu})=\lim_{n\to\infty} \frac{H(\mu^{*n})}{n}\]
For \(\mu\) of finite entropy (entropy- \(H(\kappa):=-\sum_{g} \kappa(g)\ln(\kappa(g))\) ).
\end{corollary}

\begin{proof}
Indeed, first note that one has \(H(\kappa\ast\lambda)\leq H(\kappa)+H(\lambda)\) and thus \(h=\lim\frac{H(\mu^{*n})}{n}\) exists and \(\frac{H(\mu^{*n})}{n}\geq h\). The desired equality follows from the following 3 steps:\\
\underline{\emph{Step 1}}: \(h\geq h_\mu(\nu)\) for any \(\mu\)-stationary \((X,\nu)\).\\
Let us show first \(H(\mu)\geq h_\mu(\nu)\). Indeed, \(\frac{dg^{-1}\nu}{d\nu}\geq \mu(g)\) and thus \(-\ln(\frac{dg^{-1}\nu}{d\nu})\leq -\ln(\mu(g))\) and then integrate over \(X\) and average using \(\mu\). Since for \(\mu\)-stationary \(\nu\) one has \(h_{\mu^{*n}}(\nu)=n\cdot h_{\mu}(\nu)\) we conclude step 1.\\
\underline{\emph{Step 2}}: \(h_{\mu}(\mathcal{B}(G,\mu),\nu_{\mu})=\lim_{a\to 1^{-}} h_{\mu}(G,\mu_a)\).\\
Corollary \ref{Corollary RN parameters on group and boundary} yields \(h_{\delta_g}(\mu_a)\to h_{\delta_g}(\nu_{\mu})\) for any \(g\in G\). Since \(0\leq h_{\delta_g}(\mu_a)\leq \ln(\frac{2}{\mu(g)})\) for \(a\geq\frac{1}{2}\), and \(\sum_{g}\mu(g)\ln(\frac{1}{\mu(g)})=H(\mu)<\infty\), we can replace sum and limit and conclude step 2.\\
\underline{\emph{Step 3}}:
\(h\leq\liminf_a h_{\mu}(G,\mu_a)\).\\
Indeed, one has the following formula for entropy on the group: 
\(h_{\mu}(G,\lambda)=-\sum_{g} (\mu\ast\lambda-\lambda)(g)\ln(\lambda(g))\). \\
Thus, noting that \(\mu_a(e)\geq 1-a\) and using concave of \(H\) we obtain:\\
    \(\liminf_{a} h_{\mu}(\mu_a) =\liminf_{a} -\sum_{g} (\mu\ast\mu_a-\mu_a)(g)\ln(\mu_a(g))= \liminf_a \Big(\frac{1-a}{a}H(\mu_a)+\frac{1-a}{a}\ln(\mu_a(e))\Big)\\
    \geq \liminf_{a}(1-a)^2\sum a^{n}H(\mu^{*n}) \geq \liminf_{a} (1-a)^{2}\sum_{n}a^{n}\cdot n\cdot h=h\) 

\end{proof}

The rest of the paper will consist of applications of our characterization of the Furstenberg-Poisson boundary (Theorem \ref{Theorem Abel space is Furstenberg-Poisson boundary}) and in particular its numerical consequence (Corollary \ref{Corollary RN parameters on group and boundary}). The first application will be to amenable groups (Theorem \ref{Theorem A}). The second application will be to the problem of calculating entropy minimal number
(see Definition \ref{definition minimal entropy number}) for the action of the Free group on itself (Theorems \ref{Theorem B},  \ref{Theorem C}, \ref{Theorem D}).

\section{Application for amenable groups}\label{section Application for amenable groups}
Recall that a generating probability measure on \(G\) is called \emph{Liouville} if the corresponding Furstenberg-Poisson boundary is a point.
\begin{corollary}\label{Corollary entropy and Liouville measure}
If \(\mu\) is a Liouville measure on \(G\), then for any finitely supported probability measure \(\lambda\) on \(G\) and a convex function \(f\) with \(f(1)=0\) we have:
\[\lim_{a\to 1^{-}} h_{\lambda,f}(\mu_a)=0\]
Here \(\mu_a=(1-a)\sum_{n}a^{n}\mu^{*n}\).
\end{corollary}
\begin{proof}
It is a special case of Corollary \ref{Corollary RN parameters on group and boundary}, as in this case, \(\mathcal{B}(G,\mu)\) is a point.
\end{proof}

Using the (proven) Furstenberg conjecture (see \cite{BoundaryEntropy}) we conclude the following characterization of amenability, which includes Theorem \ref{Theorem A}:

\begin{theorem}\label{Theorem application for amenability}
The following are equivalent for a discrete countable group \(G\):
\begin{enumerate}
    \item 
    \(G\) is amenable.
    \item
    For each \(S\subset G\) finite and \(\epsilon>0,f\in C(0,\infty)\) with \(f(1)=0\) there exists a symmetric probability measure \(\eta\) on \(G\) such that:
    \[\forall g\in S: \:\:\:|\intop_G f\bigg(\frac{dg\eta}{d\eta}\bigg)d\eta|<\epsilon\]
    \item
    \(G\) has KL-almost invariant symmetric measures. That is, for any finite \(S\subset G\) and \(\epsilon>0\), there is a symmetric probability measure \(\eta\) on \(G\) with 
    \[\forall g\in S: \:\:\:D_{KL}(g\eta||\eta)<\epsilon\]
    \item 
    For any QI \(G\)-space \(X\), finitely supported probability measure \(\lambda\) on \(G\) and a convex function \(f\) with \(f(1)=0\) one has \(I_{\lambda,f}(X)=0\).
\end{enumerate}
Moreover, the \(\eta\)'s in items (2),(3) can be chosen to commute.
\end{theorem}
\begin{proof}
\(\)\\
\((1)\implies(2)\): From Theorem 4.4 of \cite{BoundaryEntropy} there is a generating symmetric measure \(\mu\) such that \((G,\mu)\) is Liouville. In particular the Furstenberg-Poisson boundary is a point and thus by Corollary \ref{Corollary RN parameters on group and boundary} we conclude 
\[\lim_{a\to 1^-} \intop_G f\bigg(\frac{dg\mu_a}{d\mu_a}\bigg)d\mu_a=f(1)=0\]
and thus taking \(\eta=\mu_a\) for \(a<1\) large enough will satisfy the desired property.\\
For the moreover part, note that the measures \(\mu_a\) commute.\\ 
\((2)\implies (4)\): From \((2)\) it is clear that \(I_{\lambda,f}(G)=0\). From Corollary \ref{Cor Group-Space-Entropy} we conclude \((4)\).\\
\((4)\implies (3)\) is clear. Also \((3)\) is a special case of \((2)\) for \(f(t)=t\ln(t)\) (so the moreover part follows).\\
\((3)\implies (1)\) follows from Pinsker's classical inequality \(||m-\nu||\leq \sqrt{2D_{KL}(m||\nu)}\) (see for example Theorem 4.19 in  \cite{Boucheron2013ConcentrationI}) and Reiter's condition for amenability. \\
\end{proof}

\begin{remark}
\begin{enumerate}
    \item
    Note that in order to find \(\ell^1\)-almost invariant measures on amenable groups one can use measures supported on finite sets (for example F{\o}lner sets). However, such measures can not be the desired measures \(\eta\) given in item \((3)\) of Theorem \ref{Theorem application for amenability}. Indeed, for the KL-divergence to be finite, the measure has to be supported on the whole group (if \(G\) is finitely generated).
    \item
    If in Theorem \ref{Theorem application for amenability} one doesn't care about the commutativity (and symmetry) of \(\eta\)'s, it is possible to prove the result with the same machinery of ultralimits, without having to use Furstenberg's (proven) conjecture.\\
    This will be explained in the sequel paper \cite{ultrAmenable}, where the framework of all amenable actions is covered.
\end{enumerate}
\end{remark}
We say that a probability measure \(\eta\) on a finitely generated group \(G\) is \emph{SAS} (symmetric, adapted, smooth)  if it is symmetric, generating and \(\sum \eta(g)e^{\epsilon|g|}<\infty\) for some \(\epsilon>0\) and word metric \(|\cdot|\).\\
Using Corollary \ref{Corollary entropy and Liouville measure} one sees the following property for Liouville groups:
\begin{corollary}\label{corollary Liouville group posses SAS almost invariant measures}
Let \(G\) be a finitely generated group with a symmetric finitely supported Liouville measure. Then, there is a sequence \((\eta_n)\) of SAS measures that is \(D_f\)-almost invariant, for any \(f\)-divergence. That is, for any \(g\in G\) we have \(D_{f}(g\eta_n||\eta_n)\to0\).
\end{corollary}
\begin{proof}
This follows from Corollary \ref{Corollary entropy and Liouville measure} by noting that if \(\mu\) is finitely supported and symmetric, then \(\mu_a\) are SAS for any \(a<1\) and taking \(\eta_n=\mu_{1-1/n}\).
\end{proof}
\begin{remark}
Note that if a sequence of probability measures \((\eta_n)\) on \(G\) is uniformly BQI and \(\ell^1\)-almost invariant, then it is \(D_f\)-almost invariant for any \(f\)-divergence. Indeed, from \(\ell^1\)-almost invariance one deduces that for any non-principle ultrafilter \(\mathcal{U}\) on \(\mathbb{N}\) the ultralimit measure \(\mathcal{U}\lim \eta_n\) is \(G\)-invariant. By Lemma \ref{Lemma convergence of RN parameters} we conclude \(\mathcal{U}\lim D_f(g\eta_n||\eta_n)=D_f(g\mathcal{U}\lim\eta_n||\mathcal{U}\lim\eta_n)=0\).\\
The measures built in Theorem \ref{Theorem application for amenability} and Corollary \ref{corollary Liouville group posses SAS almost invariant measures} are uniformly BQI. 
\end{remark}
The next theorem concerns a fixed point result for amenable groups on RN models of stationary actions. It is a reformulation of a result from \cite{FixedPointAmenable}. However our proof is very different and is based on the problem of \enquote{Entropy minimizing}:

\begin{theorem}\label{Theorem fixed point for stationary actions}
Let \(G\) be an amenable group, and let \(\mu\) be a generating symmetric probability measure.
Suppose \(X\) is a compact topological space, with a continuous \(G\)-action and that \(\nu\) is a \(\mu\)-stationary measure on \(X\) such that the Radon Nikodym derivatives \(\frac{dg\nu}{d\nu}\) are continuous. \\
Then there is a continuous measure preserving factor \(\pi:X\to Y\) so that
\(Y^{G}\), the set of \(G\)-invariant points of \(Y\), contains exactly one point. In particular, \textbf{the action of an amenable group on the RN model of its Furstenberg-Poisson boundary admits a unique fixed point.}
\end{theorem}
\begin{proof}
First, by changing \(\mu\) to \(\sum_{n=1}^{\infty} 2^{-n}\mu^{*n}\), we may assume \(\mu\) is supported on all of \(G\).
By taking \(Y\) to be the RN model of the Radon Nikodym factor of \(X\), we may suppose that the Radon-Nikodym derivatives separates points of \(X\), and show that \(X^{G}\) consists of one point.\\
Denote \(E=\{x\in X\big|\:\forall g\in G:\: \frac{dg\nu}{d\nu}(x)=1\}\).\\
We claim that \(E=X^{G}\).
Indeed, \(E\) is \(G\)-invariant set, as the RN functions separates points, we see that \(E\) contains at most one point, thus \(E\subseteq X^{G}\). For the other inclusion: if \(x\in X^{G}\) then \(g\mapsto \frac{dg\nu}{d\nu}(x)\) is a multiplicative homomorphism, in particular: \(\frac{dg^{-1}\nu}{d\nu}(x)=(\frac{dg\nu}{d\nu}(x))^{-1}\), but then by the symmetry of \(\mu\) we get \(1=\sum_{g} \mu(g)\frac{dg\nu}{d\nu}(x)=\frac{1}{2}\sum_{g}\mu(g)(\frac{dg\nu}{d\nu}(x)+(\frac{dg\nu}{d\nu}(x))^{-1})\geq 
\sum_{g}\mu(g)=1\), so we have equality which implies \(\frac{dg\nu}{d\nu}(x)=1\) for any \(g\in G\), so \(x\in E\).\\
So it is enough to show \(E\neq\emptyset\). Assume for the sake of contradiction that there is no point \(x_0\) with \(\frac{dg\nu}{d\nu}(x_0)=1\) for all \(g\). By compactness of \(X\) we get there is a finite \(S_0\subset G\) so that there is no point \(x_0\) with \(\frac{dg\nu}{d\nu}(x_0)=1\) for all \(g\in S_0\). We may assume \(e\in S_0, S_0^{-1}=S_0\). Note that in the Jensen inequality 
\[-\sum_{s\in S_0} \mu(s)\ln(\frac{ds^{-1}\nu}{d\nu})\geq -\mu(S_0)\ln\Big(\sum_{s\in S_0} \frac{\mu(s)}{\mu(S_0)}\frac{ds^{-1}\nu}{d\nu}\Big)\]
There is no equality, by assumption on \(S_0\), but both sides are continuous functions, so we get some \(\epsilon_0>0\) with 
\[-\sum_{s\in S_0} \mu(s)\ln(\frac{ds^{-1}\nu}{d\nu})\geq -\mu(S_0)\ln\Big(\sum_{s\in S_0} \frac{\mu(s)}{\mu(S_0)}\frac{ds^{-1}\nu}{d\nu}\Big)+\epsilon_0\quad\quad\quad(1)\]
Let \(\delta>0\), note that there is a symmetric probability measure \(\kappa\) on \(G\) which has finite support, agrees with \(\mu\) on \(S_0\), and \(\kappa\leq (1+\delta)\cdot \mu\).
Indeed, take \(A\subset G\setminus S_0\) finite and symmetric with \(\mu(A)\geq \frac{\mu(G\setminus S_0)}{1+\delta}\) and define \(\kappa=\mu|_{S_0}+ \frac{1-\mu(S_0)}{\mu(A)}\mu|_{A}\).\\
We will show 
\[I_{\kappa}(X)\geq \epsilon_0 -\delta-\ln(1+\delta)\quad\quad\quad(*)\]
Using Theorem \ref{Theorem application for amenability} we see \(I_{\kappa}(X)=0\). This implies \(\epsilon_0\leq \delta+\ln(1+\delta)\), choosing an appropriate \(\delta>0\), we obtain a contradiction.\\
We now verify the inequality \((*)\). Using Jensen and (1):
\begin{multline*}
    -\sum_{g\in G} \kappa(g)\ln(\frac{dg^{-1}\nu}{d\nu})=-\sum_{g\in S_0} \mu(g)\ln(\frac{dg^{-1}\nu}{d\nu})-\sum_{g\in G\setminus S_0} \kappa(g)\ln(\frac{dg^{-1}\nu}{d\nu})
    \geq\\
    -\mu(S_0)\ln\Big(\sum_{s\in S_0} \frac{\mu(s)}{\mu(S_0)}\frac{ds^{-1}\nu}{d\nu}\Big)+\epsilon_0-
    \mu(G\setminus S_0)\ln\Big(\sum_{g\in G\setminus S_0} \frac{\kappa(g)}{\kappa(G\setminus S_0)}\frac{dg^{-1}\nu}{d\nu}\Big)\geq -\ln\Big(\frac{d\:\kappa\ast \nu}{d\nu}\Big)+\epsilon_0
\end{multline*}
Note that \(\kappa\ast\nu \leq (1+\delta)\mu\ast\nu=(1+\delta)\nu\) and thus \(\frac{d\kappa\ast\nu}{d\nu}\leq (1+\delta)\).\\
For any \(m\in M(X)\) (that is, in the measure class of \(\nu\)) we get by Corollary \ref{Cor KL-entropy inequality} :
\begin{multline*}
h_{\kappa}(m)\geq \intop_{X} \big(1-\frac{d\kappa \ast \nu}{d\nu}-\sum_{g\in G}\kappa(g)\ln(\frac{dg^{-1}\nu}{d\nu})\big) \:dm\geq\\
\intop_{X} \Big(1-\frac{d\kappa \ast \nu}{d\nu}- \ln\big(\frac{d\:\kappa\ast \nu}{d\nu}\big)+\epsilon_0\Big) \: dm\geq 1-(1+\delta)-\ln(1+\delta)+\epsilon_0
\end{multline*}
Taking infimum over \(m\) we obtain \(I_{\kappa}(X)\geq \epsilon_0 -\delta-\ln(1+\delta)\) and the proof is complete.
\end{proof}

\begin{remark}
One should note that for the lamplighter group, \(G=\mathbb{Z}/2\:\wr\:\mathbb{Z}^{d}\) with \(d >2 \), and the uniform symmetric measure \(\mu\) coming from the standard generating set \(\pm e_i, \delta_0\), the space \(X=\prod_{\mathbb{Z}^{d}}\mathbb{Z}/2\) of final configuration with the measure of final configuration, is a \(\mu\)-stationary space \((X,\nu)\) which is in fact the Furstenberg-Poisson boundary (see \cite{lyons2015poisson} for the Furstenberg-Poisson boundary of Lamplighter groups). However, there are no fixed points in this action. The reason is that \(X\) is not RN model. More precisely, for \(0\in X\), and \(g=e_1\), one can show that \(\frac{dg \nu}{d\nu}\) is not continuous at \(0\). We thank Gady Kozma for a discussion clarifying this issue.\\
This has the following implication: consider \(A=C^{*}\Bigg(C(X)\cup\{\frac{dg\nu}{d\nu}\}_{g\in G}\Bigg)\subset L^{\infty}(X)\) and let \(Y\) be the Gelfand dual. Then there is a continuous mapping \(\pi:Y\to X\) which is a measurable isomorphism. Moreover, \(Y\) is an RN-model, and thus by Theorem \ref{Theorem fixed point for stationary actions} we see that the \(G\)-invariant closed set \(S:=\{y\big|\:\forall g\in G:\: \frac{dg\nu}{d\nu}(y)=1\}\subset Y\) is nonempty. Since \(X\) is minimal, we conclude \(\pi(S)=X\), however \(S\) has measure zero. This means that \(X\) is covered by the null-set \(S\).

\end{remark}
\begin{remark}\label{Remark fixed point}
We believe that an existence of a fixed point on the RN-model for a Furstenberg-Poisson boundary \(\mathfrak{B}(G,\mu)\) for a symmetric generating \(\mu\) implies that \(G\) is amenable. However, we haven't been able to prove this.
\end{remark}

\section{Free groups and Entropy}\label{section Free groups and Entropy}

Let \(F=F_d\) be a free group on \(d\geq2\) generators, that we denote by \(a_1,\dots,a_d\). Denote also \(a_{-i}=a_i^{-1}\).\\
Our main goal in this section, is given \(f,\lambda\) where \(\lambda\) is symmetric and supported on the generators, to find \(I_{\lambda,f}(F):=\inf_{\nu\in M(F)} h_{\lambda,f}(F,\nu)\).

\subsection{The Boundary of the Free group and harmonic measures}\label{subsection The Boundary of the Free group and harmonic measures}

Consider the standard topological boundary \(X=\partial F\) of \(F\) as the space of infinite length reduced words in the free generators, 
that is, the subset of \(\{a_{\pm1},\dots,a_{\pm d}\}^{\mathbb{N}}\) consisting of words \((w_n)_{n\geq0}\) with \(w_{n+1}\neq w_n^{-1}\).
It is clear that \(X\) is a compact metrizable space with the standard continuous \(F\) action. \\
For any \(\gamma\in F\) we denote by \(X_\gamma\) the subspace of \(X\) consisting of words that begins with \(\gamma\). The collection \(\{X_{\gamma}\}_{\gamma\in F}\) consists of clopen sets and gives a basis for the topology of \(X\).\\
In this subsection we fix a generating probability measure \(\mu\) on \(F\) which is supported on \(\{a_{ i}\}_{i=\pm1,\dots,\pm d}\) and denote \(p_i=\mu(a_i)\).\\
\(\)\\
The next lemma and most of the following discussion can be found in section 2 of \cite{FreeBoundary}. We include the proof of the lemma in order to make our results (Theorem \ref{Theorem entropy for F action on F}) explicit.
\begin{lemma}\label{Lemma numbers q for boundary}
There exist unique \(1>q_i>0\:\: (i=\pm1,\dots,\pm d)\) such that for \(j=\pm1,\dots\pm d\):
\[q_j=p_j+q_j\sum_{i\neq j} p_i q_{-i}\]
\end{lemma}
\begin{proof}
Let \(x=1-\sum_{i} p_{i}q_{-i}\). Then the equations for \(j,-j\) are equivalent to \(\frac{q_j}{p_j}=\frac{q_{-j}}{p_{-j}}\) and \(p_{-j}q_j^{2}+xq_j-p_j=0\) which yields that \(x\) determines \(q_j\) by:
\[q_j=\frac{-x+\sqrt{x^2+4p_jp_{-j}}}{2p_{-j}}\]
Thus we need to show there exist unique \(x\in(0,1)\) with \(f(x)=0\) where:
\[f(x)=-x+1-\sum_{i}p_i \frac{-x+\sqrt{x^2+4p_{-i}p_{i}}}{2p_{i}}=(d-1)x+1-\sum_{j=1}^{d} \sqrt{x^2+4p_jp_{-j}}\]
Note that \(f(0)\geq0\) (equality for symmetric measures), \(f(1)<0\) and \(f\) is concave with \(f^{\prime}(0)=d-1>0\). Thus \(f\) has a unique zero in this interval.
\end{proof}
Consider the \(\mu\)-random walk \(R_n\) on \(F\). The random walk converges a.e. to an element of \(X\). The hitting measure \(\nu_{\mu}\) is the unique \(\mu\)-stationary measure on \(X\). The Furstenberg-Poisson boundary of \((F,\mu)\) is isomorphic to \((X,\nu_{\mu})\) (see the Theorem in section 7.4 in \cite{kaimanovichhyperbolicproperties}).\\
Before computing the measure \(\nu_{\mu}\), we make the following observations and notation:
\begin{itemize}
    \item The probabilities \(q_j:=\mathbb{P}(\exists n:\:\: R_n=a_j)\) are the numbers of the Lemma \ref{Lemma numbers q for boundary}. Indeed, they are positive and \(<1\) as the random walk is not recurrent. To show they satisfy the equation, we use the Markov property of the random walk. More precisely, we split according to the cases \(R_1=a_i\). The latter has probability \(p_i\). For \(i=j\) we get there exists such \(n\), namely \(n=1\). For \(i\neq j\), we note that in order to get to \(a_j\) we need to get first from \(a_i\) to \(e\), which has probability \(q_{-i}\), and then to \(a_j\), which has probability \(q_j\). This shows \(q_j=p_j+q_j\sum_{i\neq j} p_i q_{-i}\).
    \item
    Define \(v(a_i)=v_i=\nu_{\mu}(X_{a_i})\). We claim that \(\frac{v_i}{1-v_{-i}}=q_i\). Indeed, let us verify \(q_i(1-v_{-i})=v_i\). For this we provide a probabilistic argument: the probability that the limit \(R_\infty\) starts with \(a_{i}\) is the probability that it gets to \(a_i\) for the first time at some finite \(n_0\), and that if we consider the random walk \(R^{\prime}_n=a_i^{-1}R_{n+n_0}\) then \(R^{\prime}_\infty\) doesn't start with \(a_{-i}\). Now, by the Markov property this happens in probability \(q_{i}(1-v_{-i})\), verifying the claim.\\
    From the formulas \(\frac{v_i}{1-v_{-i}}=q_i\) and \(\frac{v_{-i}}{1-v_{i}}=q_{-i}\) we conclude the following direct formula for \(v_i\):
\[v_i=\frac{q_i(1-q_{-i})}{1-q_{-i}q_i}\]
Note that if \(\mu\) is symmetric, then \(p_i,q_i,v_i\) are symmetric and the above reduces to \(v_i=\frac{q_i}{1+q_i}\).
\end{itemize}

\begin{lemma}\label{Lemma harmonic measure formula}
The measure \(\nu_\mu\) is given by:
\[\nu_\mu(X_\gamma)= v(\gamma_1)\prod_{\ell=2}^{m} \frac{v(\gamma_\ell)}{1-v(\gamma_{\ell-1}^{-1})}\quad\quad \gamma=\gamma_1\dots\gamma_m \:\:\text{ reduced form}\]
And 
\begin{equation*}
\frac{da_{i}\nu_{\mu}}{d\nu_{\mu}}=\left\{
\begin{array}{rl}
   q_{-i} & \text{on}\:\:\:\: X\setminus X_{a_i}  \\
\frac{1}{q_i} & \text{on}\:\:\:\: X_{a_i}     
\end{array} \right.
\end{equation*}
\end{lemma}
\begin{proof}
Since \(\sum_{i} v_i=\sum_{i} \nu_{\mu}(X_{a_i})=1\) by the definition of \(v_i\), the right hand side yields a Borel probability measure \(\nu\) on \(X\). We will show 
\begin{equation*}
\frac{da_{i}\nu}{d\nu}=\left\{
\begin{array}{rl}
   q_{-i} & \text{on}\:\:\:\: X\setminus X_{a_i}  \\
\frac{1}{q_i} & \text{on}\:\:\:\: X_{a_i}     
\end{array} \right.
\end{equation*}
and conclude \(\nu\) it is \(\mu\)-stationary. By the well known uniqueness of the stationary measure we deduce \(\nu=\nu_{\mu}\).\\
Let \(\gamma=\gamma_1\dots\gamma_m\) be a reduced form. A direct computation yields:
\begin{multline*}
\big(a_i\nu\big)(X_\gamma)=\nu(a_i^{-1}X_{\gamma})=
\left\{
\begin{array}{rl}
    \nu(X_{a_i^{-1}\gamma}) & \gamma_1\neq a_i \\
    \nu(X_{\gamma_2\dots\gamma_m}) & \gamma_1=a_i
\end{array} \right.
=
\left\{
\begin{array}{rl}
    v(a_i^{-1})\frac{v(\gamma_1)}{1-v(a_i)}\prod_{j=2}^{m}\frac{v(\gamma_j)}{1-v(\gamma_{j-1}^{-1})} &  \gamma_1\neq a_i\\
v(\gamma_2)\prod_{j=3}^{m}\frac{v(\gamma_j)}{1-v(\gamma_{j-1}^{-1})}\quad\quad\quad\:\:\: & \gamma_1=a_i
\end{array} \right.\\
= \left\{
\begin{array}{rl}
    \frac{v_{-i}}{1-v_{i}} \cdot v(\gamma_1)\prod_{j=2}^{m}\frac{v(\gamma_j)}{1-v(\gamma_{j-1}^{-1})} &  \gamma_1\neq a_i\\
\frac{1-v_{-i}}{v_i}\cdot v(\gamma_1)\prod_{j=2}^{m}\frac{v(\gamma_j)}{1-v(\gamma_{j-1}^{-1})} & \gamma_1=a_i
\end{array} \right.
=\nu(X_\gamma) \cdot \left\{
\begin{array}{rl}
    q_{-i}  &  \gamma_1\neq a_i\\
\frac{1}{q_i} & \gamma_1=a_i
\end{array} \right.
\end{multline*}
This shows the formula for \(\frac{da_i \nu}{d\nu}\). Now, for \(\xi\in X\) starting with \(a_j\) we conclude that:
\[\frac{d\mu\ast\nu}{d\nu}(\xi)=\sum_{i} p_i\cdot \frac{da_i\nu}{d\nu}(\xi)=\sum_{i\neq j} p_i\cdot q_{-i} +\:\: \frac{p_j}{q_{j}}=\frac{1}{q_j}\bigg(p_j+q_j\sum_{i\neq j} p_i\cdot q_{-i}\bigg)=\frac{q_j}{q_j}=1\]
This shows \(\mu\ast\nu=\nu\) which proves the lemma. 
\end{proof}

\subsection{Measures on the boundary minimizing entropy in their class}\label{subsection Measures on the boundary minimizing in their measure class}
We now turn to finding entropy-minimizing measures on the boundary:

\begin{notation}
\begin{itemize}
    \item Let \(f\) be a strictly convex smooth function on \((0,\infty)\) with \(f(0)=0\). Denote \(\Psi_f(z)=f(z)-zf^{\prime}(z)+f^{\prime}(\frac{1}{z})\).
    \item
    Denote by \(\Delta_d\) the set of generating symmetric probability measures on \(F_d\) supported on \(\{a_i\}_{i=\pm1,\dots,\pm d}\).
\end{itemize}
\end{notation}
Note that \(\Psi_f^\prime(z)=-z\cdot f^{\prime\prime}(z)-\frac{1}{z^{2}}f^{\prime\prime}(z)<0\) and thus \(\Psi_f\) is a decreasing function.

Given \(f\) as above, we define a mapping \(\Delta_d\stackrel{T}{\to}\Delta_d\) as follows.
\begin{definition}\label{definition of T}
Given \(\mu\in\Delta_d\), denote \(p_i=\mu(a_i)\) and let \(q_i\) be as in
Lemma \ref{Lemma numbers q for boundary}. Define \(T(\mu)=\lambda\in\Delta_d\) by:
\[\lambda(a_{\pm j})=\lambda_j=\frac{c}{\Psi_f(q_j)-\Psi_{f}(\frac{1}{q_j})}\]
where \(c\) is a normalization \(c=\big(2\sum_{i=1}^{d}\frac{1}{\Psi_{f}(q_j)-\Psi_f(\frac{1}{q_j})}\big)^{-1}\).
\end{definition}
Since \(\Psi_f\) is decreasing, \(T\) is well defined.
\begin{example}
For \(f(z)=z\ln(z)\) we have \(\Psi(z)=-\ln(z)+1-z\), and \(\Psi(q)-\Psi(q^{-1})=2\ln(q^{-1})+q^{-1}-q\).
\end{example}

The next proposition gives a minimizing measure for the entropy in a specific measure classes on the boundary.
\begin{prop}\label{Proposition minimizing in the measure class}
Let \(f\) be as in the notation. Given \(\mu\in\Delta_d\) let \(\lambda=T(\mu)\).\\
Consider \(X=\partial F\) as a QI \(F\)-space with the measure class of \(\nu_{\mu}\). Then:
\[h_{\lambda,f}(X,\nu_{\mu})=I_{\lambda,f}(X)\]
\end{prop}
\begin{proof}
Note that since \(\mu\) is symmetric, \(p_i,q_i,v_i,\lambda_i\) are all symmetric.\\
Note that as \(\nu_{\mu}\) is BQI and \(\lambda\) has finite support, the conditions of Proposition \ref{Proposition f-entropy minimize} applies, and \(\nu_{\mu}\) is minimal for the entropy \(h_{\lambda,f}\) in its measure class iff \(\Psi_{\lambda,f}(\nu_{\mu};\cdot )\) is a constant function. We now show \(\Psi_{\lambda,f}(\nu_{\mu};\cdot )\) is constant.
Note that as \(\lambda\) is symmetric
\[\Psi_{\lambda,f}(\nu_{\mu};\cdot)=\sum_{i=\pm1,\dots\pm d} \lambda(a_i)\bigg(\frac{\partial F}{\partial x}(1,\frac{da_{i}\nu_{\mu}}{d\nu_{\mu}})+\frac{\partial F}{\partial y}(\frac{da_{i}\nu_{\mu}}{d\nu_{\mu}},1)\bigg)\]
Where as usual \(F(x,y):=yf(\frac{x}{y})\). Note that \[\frac{\partial F}{\partial x}(1,z)+\frac{\partial F}{\partial y}(z,1)=f(\frac{1}{z})+f(z)-z f^\prime(z)=\Psi_{f}(z)\]
Now take \(\xi\in X\) that starts with \(a_j\):
\[\Psi_{\lambda,f}(\nu_{\mu};\xi)=\sum_{i=\pm1,\dots,\pm d}\lambda(a_i)\Psi_f(\frac{d a_i\nu_{\mu}}{d\nu_{\mu}}(\xi))=C+\lambda(a_j)\cdot\big( \Psi_{f}(\frac{1}{q_j})-\Psi_{f}(q_j)\big)=C+c\]
Where \(C=\sum_{i=\pm1,\dots,\pm d} \lambda(a_i)\Psi_f(q_i)\) is independent of \(j\), and \(c\) is the normalization in the definition of \(T\) (see Definition \ref{definition of T}).
This is independent of \(j\), concluding the proof.
\end{proof}

\begin{remark}
Note that in the infimum \(I_{\lambda,f}(X)\) we only take measures in the measure class of \(\nu_{\mu}\). In a future work, the first named author discusses the problem of minimal entropy over all measures on \(X\), and it will turn out that the infimum remains the same, although this requires considerable more effort.
\end{remark}

We would next like to show that we indeed cover in our analysis \textbf{all} the generating measures \(\lambda\) which are symmetric and supported on the set of free generators and their inverses.

\begin{lemma}\label{Lemma mu to lambda is bijection}
The mapping \(T:\Delta_d\to \Delta_d\) is a bijection.
\end{lemma}
\begin{proof}
We first extend \(T\) to the boundary of the simplex \(\Delta_d\). This can be done by interpreting the boundary as elements of \(\Delta_{r},\:r<d\) and taking the corresponding \(T\) there. This yields a continuous map \(T:\overline{\Delta_d}\to\overline{\Delta_d}\) that sends the boundary to itself. This implies that \(T(\Delta_d)=\Delta_d\cap T(\overline{\Delta_d})\subset \Delta_d\) is relatively closed.\\
We wish to show that the image is also open, which would then imply \(T\) is surjective. By invariance of domain (cf. Theorem 2B.3 in \cite{Hatcher}) it is enough to show that \(T\) is injective. \\
In order to show injectivity of \(T\), we first show that \(\mu=(p_i)\mapsto (q_i)\) is injective, and then that \((q_i)\mapsto\lambda\) is injective.\\
For the first mapping, from \((q_i)\) one can reconstruct \((p_i)\), as the defining equation is:
\[\begin{pmatrix}
\frac{1}{q_1}+q_1 & 2q_2 & 2q_3 & \dots & 2q_d \\
2q_1 & \frac{1}{q_2}+q_2 & 2q_3 & \dots & 2q_d \\
& & \vdots\\
2q_1 & 2q_2 & 2q_3 &\dots & \frac{1}{q_d}+q_d
\end{pmatrix} 
\begin{pmatrix}
p_1 \\
p_2 \\
\vdots\\
p_d
\end{pmatrix}=
\begin{pmatrix}
1 \\
1 \\
\vdots\\
1
\end{pmatrix}\]
So we will show the matrix above is non-singular. Employing row operations, this matrix has the same determinant as the matrix:
\[\begin{pmatrix}
\frac{1}{q_1}+q_1 & 2q_2 & 2q_3 & \dots & 2q_d \\
q_1-\frac{1}{q_1} & \frac{1}{q_2}-q_2 & 0 & \dots & 0 \\
\vdots & \vdots & \vdots &\dots & \vdots\\
q_1-\frac{1}{q_1} & 0 & 0 &\dots & \frac{1}{q_d}-q_d
\end{pmatrix} \]
We show that the determinant of the latter matrix is positive by induction. Indeed, expanding using the second row, we get determinant:
\[-(q_1-\frac{1}{q_1})2q_2(\frac{1}{q_3}-q_3)\dots(\frac{1}{q_d}-q_d)+(\frac{1}{q_2}-q_2)\det\begin{pmatrix}
\frac{1}{q_1}+q_1 & 2q_3 & \dots & 2q_d \\
q_1-\frac{1}{q_1} & \frac{1}{q_3}-q_3  & \dots & 0\\
& \vdots  &\\
q_1-\frac{1}{q_1}  & 0 &\dots & \frac{1}{q_d}-q_d
\end{pmatrix}\]
The second term is the same determinant but for less variables and since \(0<q_i<1\) the first term is \(>0\), by the induction assumption the second term is positive. So we need to show only the base case, and indeed 
\[\det\begin{pmatrix}
\frac{1}{q_1}+q_1 & 2q_2 \\
q_1-\frac{1}{q_1} & \frac{1}{q_2}-q_2 
\end{pmatrix}
=(\frac{1}{q_1}+q_1)(\frac{1}{q_2}-q_2)+2q_2(\frac{1}{q_1}-q_1)>0\]
Thus \(\mu\mapsto (q_i)\) is injective. Let us show that the mapping \((q_i)\mapsto \lambda\) is injective.\\ 
Consider \(\varphi(z)=\Psi_f(z)-\Psi_f(\frac{1}{z})\) this is a positive decreasing function (as \(\Psi_f\) is decreasing). Note that from \(\lambda\) we can reconstruct \(\frac{\varphi(q_i)}{\varphi(q_j)}=\frac{\lambda_j}{\lambda_i}\) for each \(i,j\). Assume that \(\lambda\) coresponds to distinct \((q_j)\neq(q^{\prime}_j)\). Without loss of generality, there is some \(j_0\) with \(q_{j_0}>q^{\prime}_{j_0}\). As \(v(z)=\frac{z}{1+z}\) is an increasing function and \((v(q_j)),(v(q^{\prime}_j))\) are probability measures (see the discussion before \ref{Lemma harmonic measure formula}), we conclude that there must be \(j_1\) with \(q_{j_1}<q_{j_1}^{\prime}\). But then \(\frac{\varphi(q_{j_0})}{\varphi(q_{j_1})}<\frac{\varphi(q^{\prime}_{j_0})}{\varphi(q^{\prime}_{j_1})}\), which is a contradiction.
\end{proof}

\begin{remark}
Note that the proof of Proposition \ref{Proposition minimizing in the measure class} yields that for \(\lambda\neq T(\mu)\) the measure \(\nu_{\mu}\) is not minimizing in its measure class for \(h_{\lambda,f}\).\\
It is not difficult to see that for \(f(t)=t\ln(t)\), the map \(T\) is not the identity (see the following example). This means that the \(\mu\)-entropy minimizer on the boundary \(X\) is, in general, \textbf{not} the \(\mu\)-stationary measure.
\end{remark}
\begin{example}
Let \(d=2\) and \(f(t)=t\ln(t)\). Consider \(\mu(a_{\pm1})=\frac{1}{3}, \mu(a_{\pm2})=\frac{1}{6}\). A direct computation (using computer and the proof of Lemma \ref{Lemma numbers q for boundary}) shows that \(\lambda=T(\mu)\) has \(\lambda(a_{\pm1})\sim 0.32378\) and in particular \(\lambda\neq\mu\). Further computation (using Lemma \ref{Lemma harmonic measure formula}) yields \(h_{\lambda}(\nu_{\mu})\sim 0.5126\).
\end{example}

\subsection{Minimizing entropy for the action of F on itself}\label{subsection Entropy minimize of the action of F on itself}
In the following theorem, we compute the minimal entropy number \(I_{\lambda,f}\) for the action of the free group \(F\) on itself. Note the next theorem includes Theorem \ref{Theorem C}:
\begin{theorem}\label{Theorem entropy for F action on F}
Let \(f\) be a strictly convex smooth function on \((0,\infty)\).\\
Let \(\lambda\in\Delta_d\) and let \(\mu=T^{-1}(\lambda)\in\Delta_d\). Then:
\[I_{\lambda,f}(F)=h_{\lambda,f}(\partial F,\nu_{\mu})\]
\end{theorem}
\begin{proof}
Consider \(X=\partial F\) as a QI \(G\)-space with the measure class of \(\nu_{\mu}\), then by Corollary \ref{Cor Group-Space-Entropy} we have
\(I_{\lambda,f}(F)\geq I_{\lambda,f}(X)\). By Proposition \ref{Proposition minimizing in the measure class} we have \(I_{\lambda,f}(X)=h_{\lambda,f}(X,\nu_{\mu})\). \\
On the other hand, by Corollary \ref{Corollary RN parameters on group and boundary} (the corollary of the main result):
\[h_{\lambda,f}(X,\nu_{\mu})=\lim_{a\to 1^{-}} h_{\lambda,f}(F,\mu_a)\geq I_{\lambda,f}(F)\]
This finishes the proof.
\end{proof}
As for Theorem \ref{Theorem B} and the more general Theorem \ref{Theorem D}:
\begin{corollary}\label{Corollary case of uniform measure}
Let \(\mu\) be the uniform symmetric measure \(\mu(a_{\pm j})=\frac{1}{2d}\). Then for any strictly convex and smooth \(f\) we have:
\[I_{\mu,f}(F_d)=\frac{2d-1}{2d} f(\frac{1}{2d-1})+\frac{1}{2d}f(2d-1)\]
In particular, taking \(f(t)=t\ln(t)\) we conclude:
\[I_{\mu,KL}(F_d)=\frac{d-1}{d}\ln(2d-1)\]
\end{corollary}
\begin{proof}
It is clear that \(v_i,q_i\) are independent of \(i\) and thus \(v_i=\frac{1}{2d},\lambda=T(\mu)=\mu\).\\
Hence, \(q_i=\frac{\frac{1}{2d}}{1-\frac{1}{2d}}=\frac{1}{2d-1}\). Thus by Theorem \ref{Theorem entropy for F action on F}:
\[I_{\mu,f}(F_d)=h_{\lambda,f}(\nu_{\mu})=\sum_{i}\frac{1}{2d}\intop_{\partial F_d}f(\frac{da_i\nu_{\mu}}{d\nu_{\mu}})d\nu_{\mu}=\frac{2d-1}{2d} f(\frac{1}{2d-1})+\frac{1}{2d}f(2d-1)\]
In particular:
\(I_{\mu,KL}(F_d)
=\frac{d-1}{d}\ln(2d-1)\).
\end{proof}

\section{Final remarks}

\subsection{Possible relation with the Liouville problem}\label{subsection Possible relation with the Liouville problem}
\subsubsection{Entropy and stationary measures}
One initial motivation for this work was an attempt to present stationary measures as information theoretic minimizers.
Using the approach and results established in this paper, where one approximates the entropy of all boundaries by the self action of the group, it is not difficult to see that this would have immediate implications to the fundamental
open problem of (in)dependence of the Liouville property on the (symmetric, finitely supported) generating measure chosen on the group.
A natural approach towards the minimization property would be to show that convolution with \(\lambda\) might reduce the \(\lambda\)-entropy. 
As it turns out, this is usually not true. Indeed, for the action \(F_2\curvearrowright\partial F_2\) and \(\lambda \) a non uniform measure on the free generators, the \(\lambda\)-harmonic measure is not in general a minimizing measure for the \(\lambda\)-entropy. Moreover, there is a minimum (in a measure class) for the entropy which is the \(\mu\)-harmonic measure for a measure \(\mu\neq\lambda\). 
In fact, let us explain why in general no information theoretic entropy can be used to detect stationary measures as minimizers.\\
Let \(h=h_{\kappa,f}\) (with \(\kappa\) finitely supported) be some entropy for an amenable group \(G\) with a non-Liouville measure \(\mu\) (e.g. \(G=\mathbb{Z}/2\:\wr\:\mathbb{Z}^{3}\)). We explain why convolution by \(\mu\) can not always decrease the entropy \(h\).\\
Indeed, take \((B,\nu)\) to be (a topological model of) the Furstenberg-Poisson boundary. We know \(0=I_{\kappa,f}(B)\) and hence there is \(m\in M(B)\) (in the measure class) with \(h(m)<h(\nu)\). Observe that \(\mu^{*n}\ast m\) converges weakly to \(\nu\).
Indeed, consider \(\Omega=G^{\mathbb{N}}\) with the \(\mu\)-random \((X_n)\) and its probability measure \(\mathbb{P}\). Note that for the stationary \(\sigma\)-algebra \(\mathcal{S}\) (see subsection 0.3 of \cite{BoundaryEntropy}) one has that \((B,\nu)\) is a topological model of \((\Omega,\mathcal{S},\mathbb{P})\). Suppose \(h=\frac{dm}{d\nu}\), then for any \(a\in L^{\infty}(B)=L^{\infty}(\Omega)\) we have:
\begin{multline*}
(\mu^{*n}\ast m)(a)=\mathbb{E}_{\mathbb{P}}\Big[(X_n\cdot m)(a)\Big]=\mathbb{E}_{\mathbb{P}}\Big[\intop_{B}a(X_n\cdot x)h(x)d\nu(x)\Big]=\intop_{B}\mathbb{E}_{\mathbb{P}}\big[a(X_n\cdot x)\big]h(x)d\nu(x)
\end{multline*}
Note that (a.e.) \(a(X_{n}(\omega)\cdot x)\to a(\omega)\). Thus \(\mathbb{E}_{\mathbb{P}}\big[a(X_n\cdot x)\big]\to \nu(a)\) for a.e \(x\in B\) which implies that \((\mu^{*n}\ast m)(a)\to \nu(a)\) which shows weak convergence.\\
Since \(h\) is lower semi-continuous w.r.t weak topology (follows from Theorem 2.34 in \cite{gariepy2001functions} which shows \(f\)-divergence are such) we conclude that \(\liminf_{n} h(\mu^{*n}\ast m)\geq h(\nu)>h(m)\), showing that convolution by \(\mu\) cannot always decrease entropy.


\subsubsection{Other versions of almost-invariant measures}

In view of Theorem \ref{Theorem application for amenability}, one can ask when there are \(\ell^\infty\)-almost-invariant measures on a group. 
That is, when for every \(S\subset G\) finite and \(\epsilon>0\), there is a measure \(\eta\) on \(G\) such that \(||\frac{ds\eta}{d\eta}-1||_{\infty}\leq \epsilon\) for all \(s\in S\).
It is easy to see that for finitely generated groups this is equivalent to having sub-exponential growth.\\
\(\)\\
Recall that we say that a measure \(\kappa\) on a finitely generated group \(G\) is \emph{SAS} if it is symmetric, generating and satisfies \(\sum \kappa(g)e^{\epsilon|g|}<\infty\) for some \(\epsilon>0\) and word metric \(|\cdot|\).
\begin{definition}
A finitely generated group is \emph{tame} if there is a sequence \((\kappa_n)\) of SAS measures that is KL-almost invariant.
\end{definition}
Corollary \ref{corollary Liouville group posses SAS almost invariant measures} implies that groups with a symmetric finitely supported Liouville measure are tame. We find the following problem intriguing:
\begin{problem}
Is there a finitely generated amenable group which is not tame?\\
Are the lamplighters \(\mathbb{Z}/2\:\wr\:\mathbb{Z}^{d},\: d\geq 3\) tame?
\end{problem}
Our hope is that the above lamplighters are not tame. One can then make the speculation that being tame is related to the  
\enquote{Liouville dependence problem} concerning boundary triviality of symmetric finitely supported measures.
Our proof of the fixed point Theorem \ref{Theorem fixed point for stationary actions} came as an attempt to employ these ideas, However eventually it produces a fixed point without the SAS assumption.




\subsection{Uniqueness of entropy minimizers}\label{subsection Entropy on other actions}
As explained in the introduction, for some actions \(G\curvearrowright X\), the infimum of the entropy function is attained.
The f-divergence is lower semi-continuous with respect to the weak topology\footnote{see Theorem 2.34 in \cite{gariepy2001functions}, there is also a proof using a variational representation of \(f\)-divergences.}. Thus if \(X\) is a compact \(G\)-space, the entropy function \(h_{\lambda,f}: M(X)\to [0,\infty]\) is lower semi-continuous.
Hence the entropy \(h_{\lambda,f}\) attains its infimum \(I_{\lambda,f}^{top}(X)\). It is then very natural to inquire about its uniqueness.
While as we have seen (Lemma \ref{Lemma uniqueness of entropy minimize}), inside a specific ergodic measure class, uniqueness holds under natural conditions on \(f\), we have no criterion for uniqueness in the topological setting, nor can we establish any non trivial uniqueness result. \\
Another problem is finding a criterion for a measure \(\nu\) to be a minimizer of entropy. Inside a specific measure class we have Proposition \ref{Proposition minimizing in the measure class} yielding a complete solution, in the topological case however, we do not have any criterion.\\
In the sequel paper \cite{ultrAmenable} we will show that for the topological action \(F_d\curvearrowright\partial F_d\) we have under the notations of Proposition \ref{Proposition minimizing in the measure class} that \(I_{\lambda,f}^{top}(\partial F_d)=h_{\lambda,f}(\partial F_d,\nu_{\mu})\).





\printbibliography

\end{document}